\documentclass[12pt,reqno]{amsart}
 \usepackage[margin=1in]{geometry} 
\usepackage{amsmath,amsthm,amssymb,amsfonts, wasysym, caption,color}
 \usepackage[draft]{optional}
\usepackage{graphicx,float}

\usepackage[colorlinks, 
citecolor=darkgreen, linkcolor=darkred
]{hyperref} 

\definecolor{darkgreen}{rgb}{0,0.5,0}
\definecolor{darkred}{rgb}{0.7,0,0}

\newcommand{\N}{\mathbb{N}}
\newcommand{\Z}{\mathbb{Z}}
 
 \newcommand{\R}{\mathbb{R}}
\newcommand{\eps}{\varepsilon}
\renewcommand{\th}{\theta}
\newcommand{\vphi}{\varphi}
\newcommand{\half}{\tfrac{1}{2}}

\newcommand{\de}{\delta}
\newcommand{\al}{\alpha}
\newcommand{\be}{\beta}
\newcommand{\Om}{\Omega}
\newcommand{\om}{\omega}
\newcommand{\la}{\lambda}
\newcommand{\La}{\Lambda}
\newcommand{\LP}{{\mathrm{P}}}

\DeclareMathOperator*{\cosec}{cosec}

\let\d\relax
\newcommand{\d}{\partial}

\newcommand{\x}{{\bar x}}

\newcommand{\Gcone}{{\Gamma_{\textup{cone}}}}
\newcommand{\Gball}{{\Gamma_{\textup{ball}}}}
\newcommand{\gsymm}{{\gamma_{\textup{symm}}}}
\newcommand{\gcone}{{\gamma_{\textup{cone}}}}
\newcommand{\gball}{{\gamma_{\textup{ball}}}}
\newcommand{\dist}{{\textrm{dist}}}

\newcommand{\beq}{\begin{equation}}
\newcommand{\eeq}{\end{equation}}
\newcommand{\beqs}{\begin{equation*}}
\newcommand{\eeqs}{\end{equation*}}
\newcommand{\beqa}{\begin{equation}\begin{aligned}}
\newcommand{\eeqa}{\end{aligned}\end{equation}}
\newcommand{\beqas}{\begin{equation*}\begin{aligned}}
\newcommand{\eeqas}{\end{aligned}\end{equation*}}

\theoremstyle{plain}
\newtheorem{lemma}{Lemma}[section]
\newtheorem{thm}[lemma]{Theorem}

\theoremstyle{definition}

\newtheorem{rmk}[lemma]{Remark}

\numberwithin{equation}{section}

\title[Elliptic equations on conical domains]{{\sc
Oblique derivative problems for elliptic equations on conical domains}
\\ 
}
\author[Matthew R. I. Schrecker]{Matthew R. I. Schrecker}\thanks{Department of Mathematics, University College London, London, UK\\
\indent Email: m.schrecker@ucl.ac.uk}
\begin{document}
 \maketitle
 \begin{abstract}
 We study the oblique derivative problem for uniformly elliptic equations on cone domains. Under the assumption of axi-symmetry of the solution, we find sufficient conditions on the angle of the oblique vector for H\"older regularity of the gradient to hold up to the vertex of the cone. The proof of regularity is based on the application of carefully constructed barrier methods or via perturbative arguments. In the case that such regularity does not hold, we give explicit counterexamples. We also give a counterexample to regularity in the absence of axi-symmetry. Unlike in the equivalent two dimensional problem, the gradient H\"older regularity does not hold for all axi-symmetric solutions, but rather the qualitative regularity properties depend on both the opening angle of the cone and the angle of the oblique vector in the boundary condition. 
 \end{abstract}
 
 \section{Introduction}
 
 The aim of this paper is to study the regularity up to the boundary for solutions of oblique derivative problems for uniformly elliptic equations on domains with conical singularities. As well as being of interest from the point of view of elliptic PDE theory, oblique derivative problems on cone domains arise naturally in a range of important physical problems, such as shock reflection problems in gas dynamics. The basic H\"older regularity of solutions has been known since the work of Miller \cite{Miller}, who initially derived suitable barrier functions on cone domains. In general, on domains with cone singularities, one cannot expect that the gradient of the solution will remain H\"older continuous up to the boundary of the domain, but under a symmetry assumption on the solution (axi-symmetry), the situation becomes more subtle. In fact, as we will show below, for these symmetric solutions, the gradient H\"older regularity (or lack of it) depends on a relationship between the opening angle of the cone at its vertex  and the angle of the oblique vector in the boundary condition. Such a relationship is somewhat surprising given the theory for the equivalent two-dimensional problem (where the cone is replaced by a wedge). For such problems, reflectional symmetry of the solution is sufficient to guarantee the H\"older regularity of the gradient up to the vertex with no restriction on the angle of the oblique vector (beyond the necessity of obliqueness). This regularity theory for symmetric solutions of oblique problems in two dimensions played a crucial role in the recent resolution of the shock reflection problem for potential flow past a wedge by Chen and Feldman \cite{CF}, and so the results contained in this paper may be expected to play a similarly important role in the three dimensional theory.
 
 To fix ideas, we work with the equation
 \beq\label{eq:fullspace}
 \begin{cases}
 Lu:=A^{ij}\d_{ij}u+A^{i}\d_iu+A^0u=f &\text{ in }\Om,\\
 Mu:=\be\cdot Du+\be^0u=g & \text{ on }\d\Om,
 \end{cases}
 \eeq
 where $\Omega\subset\R^n$ is a bounded, Lipschitz domain satisfying an exterior cone condition at every point of its boundary, $A^{ij},A^i,A^0:\Om\to\R$. We assume the existence of constants $\la,\La>0$ such that the principal coefficients $A^{ij}$ satisfy the uniform ellipticity assumption
 \beqa
 \la|\xi|^2\leq A^{ij}(x)\xi_i\xi_j\leq \La|\xi|^2 \text{ for all }\xi\in\R^n,\:x\in\Om.
 \eeqa
 Further regularity assumptions on the coefficients will be stated below in Theorem \ref{thm:main}. The vector field $\be$ is assumed to be piecewise smooth on $\d\Om$, inward pointing and oblique at all of its points of continuity. Throughout this paper, we use the following sense of the term obliqueness: a vector $\be$ is said to be oblique at a point $x_0\in\d\Om$ if there is an orthonormal coordinate system $(x',x_n)$ based at $x_0$, a positive radius $\rho>0$, and a Lipschitz function $F$ such that
 $$\Om\cap B_\rho(x_0)=\{x_n>F(x')\,|\,|x|<\rho\}$$
 such that $\be$ is parallel to the $x_n$ axis. Note that this definition may be weakened in certain directions, see for example the book of Lieberman \cite{L}.
 
 As the main issue that this paper will be concerned with is the regularity at cone points of the boundary for solutions satisfying a rotational symmetry, we will assume for simplicity of notation that $\Om$ is the intersection of a ball of fixed radius $R>0$ with a fixed cone $\mathcal{C}$ with vertex at the origin and axis of symmetry the $x_n$ axis. Following the terminology of Miller \cite{Miller}, we define polar coordinates such that 
 $$r=|x|,\quad r\cos\th=x_n,$$
 and then let the open cone be
 $$\mathcal{C}=\{(r,\th)\in(0,\infty)\times[0,\th_0)\}$$
for some fixed $\th_0\in(0,\pi)$.  Our domain is then 
$$\Om=\mathcal{C}\cap B_R(0)=\{(r,\th)\in(0,R)\times[0,\th_0)\}.$$
In general, we will write 
$$\Om[\rho]=\Om\cap B_\rho(0)\quad \text{ for }0<\rho<R.$$
We write $\Gcone$ and $\Gball$ for the portions of $\d\Om$ as follows:
\beqs
\Gcone=\{0<r<R,\,\th=\th_0\},\quad \Gball=\{r=R,\,\th\in[0,\th_0)\}.
\eeqs
Note that both portions are relatively open, so that $\d\Om=\overline\Gcone\cup\overline\Gball$.

The assumption of axi-symmetry that we will make requires certain compatibility conditions on the coefficients in order that the rotational solution can exist in the first place. For given examples, this is typically easy to compute in cylindrical coordinates. The cylindrical coordinates are the following:
$$y_1=x_n, \quad y_2=|x'|,\quad \phi\in \mathbb{S}^{n-2},$$
where $\phi$ is a standard coordinate system on $\mathbb{S}^{n-2}$. Note then that $(r,\th)$ coincide with the polar coordinates on the half space $\{(y_1,y_2)\,|\,y_2\geq 0\}$: $r=|(y_1,y_2)|$ and
$$y_1=r\cos\th,\quad y_2=r\sin\th.$$
 Axi-symmetry of the solution means that $u=u(y)$ depends only on the axi-symmetric coordinates $(y_1,y_2)$, i.e.~is independent of $\phi$, and we may therefore work equivalently on the two-dimensional domain 
$$\om=\om[R]=\{y\,|\,r(y)\in(0, R),\,\,\th(y)\in[0,\th_0)\}.$$
We define three (relatively open) boundary portions for $\om$:
$$\gcone=\{r\in(0,R),\th=\th_0\},\quad\gsymm=\{r\in(0,R),\th=0\},\quad\gball=\{r=R,\th\in(0,\th_0)\}.$$
In the $y$ coordinates, we obtain the equation
\beq
\begin{cases}
a^{ij}\d_{ij}u+b^i\d_iu+cu=f &\text{ in }\om,\\
\be\cdot Du+\be^0u=g &\text{ on }\gcone\cup\gball,\\
u_{y_2}=0 &\text{ on }\gsymm,
\end{cases}
\eeq
where the boundary operator $\be$, $\be^0$ (assumed axi-symmetric) is defined in the obvious way,
\beqa
&a^{11}=A^{nn},\quad
a^{12}=a^{21}=\sum_{i=1}^{n-1}A^{in}\frac{x_i}{y_2},\quad
a^{22}=\sum_{i,j=1}^{n-1}A^{ij}\frac{x_ix_j}{y_2^2},\\
&b^1=A^n,\quad b^2=\frac{1}{y_2}\big(\sum_{i=1}^{n-1}A^{ii}-\sum_{i,j=1}^{n-1}A^{ij}\frac{x_ix_j}{y_2^2}\big)+\sum_{i=1}^{n-1}A^i\frac{x_i}{y_2},\quad c=A^0.
\eeqa
Uniform ellipticity of the coefficients $a^{ij}$ is inherited from that of $A^{ij}$. Indeed, the assumption that \eqref{eq:fullspace} admits a rotationally symmetric solution allows us to take the coefficients $a^{ij}$ to be independent of $\phi$ and the ellipticity is straightforward to check. A significant role in our analysis will be played by the singular coefficient $b^2$, arising from the use of symmetry to reduce the problem to a two-dimensional one. In general, coefficients of the form $b^i$ satisfying only a bound $|b^i|\leq \frac{C}{y_2}$ are not suitable for the application of the methods in this paper. For the treatment of equations with coefficients of this degree of singularity for the Dirichlet problem, see the work of Fichera \cite{Fichera}, Lieberman \cite{L08}, Michael \cite{Michael} and the references therein. That this coefficient arises from a symmetry reduction of dimension is crucial. We therefore split the first order (singular) coefficient $b^2$ into two pieces:
\beq
b^{2,1}=\sum_{i=1}^{n-1}A^{ii}-\sum_{i,j=1}^{n-1}A^{ij}\frac{x_ix_j}{y_2^2},\quad b^{2,2}=b^2-\frac{b^{2,1}}{y_2}.
\eeq
We see that only the principal part, $b^{2,1}$, depends on the principal coefficients of the equation and that $b^{2,1}$ and the remainder, $b^{2,2}$, satisfy the estimates
\beq
0<(n-2)\frac{\la}{y_2}\leq \frac{b^{2,1}}{y_2}\leq (n-2)\frac{\La}{y_2},\quad |b^{2,2}|\leq \Big(\sum_{i=1}^{n-1}A_i^2\Big)^{\frac{1}{2}}.
\eeq 
We also make the technical assumption, satisfied by many physically motivated problems, that at the cone vertex, the constant coefficient operator $\overline{L}_0$ defined on $\Omega$ by
\beq\label{eq:L0bar}
\overline{L}_0=A^{ij}(0)\d_{ij}\quad \text{ is invariant under rotations around the axis of symmetry}.
\eeq
The oblique vector $\be$ is always taken to be inward pointing and (without loss of generality) such that
$$\lim_{\substack{y\to0\\y\in\gcone}}\be=(\cos(s),\sin(s)) \text{ for some }s\in(-\pi+\th_0,\th_0).$$
Note that this limit only makes sense in the $y$ coordinates as $\be$ is generally discontinuous at the origin when considered on $\Gcone$ (for example, the unit normal to $\Gcone$).

We are now in a position to state a rough version of our main theorem. 

\begin{thm}\label{thm:rough}
Suppose that the coefficients $A^{ij}$, $A^i$, $A^0$ of \eqref{eq:fullspace} satisfy \eqref{eq:L0bar} and are H\"older continuous in $\Om$, that $\be$ and $\be^0$ are H\"older continuous on each of the smooth portions of $\d\Om$ with $\be$ uniformly oblique and inward pointing, and that the problem admits an axi-symmetric solution $u$ in the sense described above. Finally, suppose $f$ and $g$ are H\"older continuous. Then there exists $s_1\in(-\frac{\pi}{2},0)$ such that if 
$$s\in\big((-\pi,-\frac{\pi}{2})\cup(s_1,\frac{\pi}{2})\big)\cap(-\pi+\th_0,\th_0),$$ 
 then the gradient of $u$ satisfies an \textit{a priori} H\"older estimate up to the vertex of the cone in terms of the H\"older norms of the data and coefficients.
\end{thm}

The requirement $s\in(-\pi+\th_0,\th_0)$ is equivalent to the choice of $\be$ as inward-pointing. However, the restriction on $s$ within this interval is necessary, due to the following theorem.

\begin{thm}\label{thm:counterexample}
Consider the oblique derivative problem for the Laplace equation on a cone:
\beq\label{eq:laplace}
\begin{cases}
\Delta u=0 &\text{ in }\Om,\\
\be\cdot Du=0 &\text{ on }\Gcone,
\end{cases}
\eeq 
where $\beta=\sin(s)\d_{r'}+\cos(s)\d_{x_n}$, $r'=|x'|$, corresponds to a constant vector $\be=(\cos(s),\sin(s))$ on $\gcone$, $s\in(-\pi+\th_0,\th_0)$.
\begin{itemize}
\item[(i)] Suppose that $\th_0\in(\frac{\pi}{2},\pi)$. There exists $s_0\in(-\pi+\th_0,0)$ depending on $\th_0$ such that if either $s\in(-\pi+\th_0,s_0)$ or $s\in(\frac{\pi}{2},\th_0)$, then there exists a H\"older continuous axi-symmetric solution of \eqref{eq:laplace} which is not $C^1$ up to the origin.
\item[(ii)] Suppose that $\th_0\in(0,\frac{\pi}{2})$. There exists $s_0\in(-\frac{\pi}{2},0)$ depending on $\th_0$ such that if $s\in(-\frac{\pi}{2},s_0)$, then there exists a H\"older continuous axi-symmetric solution of \eqref{eq:laplace} which is not $C^1$ up to the origin.
\end{itemize}
\end{thm}
As we will see in \S\ref{sec:counter2}, these solutions are smooth away from the vertex; the loss of regularity is due to the angle of the oblique vector at the vertex and the opening angle of the cone.

To provide some context for these results, we compare the situation to that for the two-dimensional problem. In two dimensions, the equation is posed on the exterior of a wedge,
$$\tilde\Om=\{r\in(0,R),\,\th\in(-\th_0,\th_0)\},$$
with $(r,\th)$ being polar coordinates on $\R^2$. The coefficients and data are taken to be H\"older continuous as in Theorem \ref{thm:rough} above. The corresponding gradient regularity result for solutions with reflectional symmetry ($u(x_1,x_2)=u(x_1,-x_2)$) is then that for all $\th_0\in(0,\pi)$, there exists $\al\in(0,1)$ depending on $\th_0$ such that $u\in C^{1,\al}$ up the vertex, see \cite[Chapter 4]{CF} or \cite[Section 4.5]{L}. The key point here is that gradient H\"older regularity holds for all uniformly oblique vectors (again, only in the case of symmetric solutions). This is a stark contrast to the situation under consideration at present, where the regularity depends on a relationship between the opening angle of the cone and the oblique vector. 

Oblique derivative problems for elliptic equations arise naturally in a wide variety of physical situations such as the theory of reflected shock waves in transonic flow and the capillary problem. In the theory of transonic shocks, in certain circumstances, one can pose the Rankine-Hugoniot conditions across the shock as an oblique derivative condition for a potential function, for example in two dimensions by \v{C}ani\'c, Keyfitz and Lieberman, \cite{CKL}. Such a formulation of the Rankine-Hugoniot conditions is also used for the shock reflection problem for potential flow, solved recently by Chen and Feldman \cite{CF} in two dimensions. The capillary problem has also been studied widely, see for example the monograph of Finn, \cite{Finn}. 

Extensions of these results to three-dimensional domains require a more detailed understanding of the regularity of oblique derivative problems on domains such as cones. In particular, the results in this paper are suitable for application to the three-dimensional shock reflection problem from a cone. For this and other physical problems, the symmetry conditions that we impose here to study the oblique derivative problem are very natural. As mentioned above, if we do not have such symmetry assumptions, we cannot, in general, expect the H\"older regularity of the gradient. See Appendix \ref{sec:counter1} for a discussion of this and the construction of a counterexample.

From the point of view of the analysis of oblique derivative problems for elliptic equations, we mention in particular the early work of Fiorenza \cite{Fiorenza}. The theory of such problems was taken up in a series of papers by Lieberman, of which the most significant for our purposes here are \cite{L87,L88} which provide gradient H\"older regularity under the stronger condition that either the smooth portions of the boundary $\d\Om$ meet along co-dimension 2 hypersurfaces or that the vector field $\be$ is continuous (see also Lieberman's monograph \cite{L} which contains many details of the theory of oblique problems in a variety of settings). 

Concerning the H\"older regularity of the solution $u$, for domains with conical singularities, Miller \cite{Miller} constructed a barrier function for such regularity theory for the Dirichlet problem. We mention also the early result of Nadirashvili \cite{N} on smooth domains for the oblique derivative problem and the more recent work of Nadirashvili and Kenig \cite{KN} for oblique derivative problems on smooth domains with source terms $f\in L^p$. In general, we refer to \cite{L01} for the Harnack inequality and pointwise estimates (including H\"older estimates) of solutions of such problems on Lipschitz (or less regular) domains. For H\"older estimates for viscosity solutions of fully non-linear Neumann problems, we refer to the paper of Barles and Da Lio, \cite{BDL}.  The theory of elliptic equations with strongly singular lower order terms (such as the $y_2^{-1}$ term that we find here in the axi-symmetric coordinates) with Dirichlet data has been studied by Fichera \cite{Fichera} from the point of view of degenerate equations and also by Lieberman \cite{L08}, generalising the earlier work of Michael \cite{Michael}.

A precise statement of Theorem \ref{thm:rough} will be given below in \S\ref{sec:mainproofs} in two parts: Theorems \ref{thm:main} and \ref{thm:perturb}. The first of these two results covers the case $s\in(0,\frac{\pi}{2})\cup(-\pi,-\frac{\pi}{2})$, while the second is for $s\in(s_1,0]$. These theorems give precise \textit{a priori} H\"older estimates for the gradient of the solution $u$. We outline the strategy of proof for Theorem \ref{thm:main} as follows. As the intersection of the boundaries of the ball and the cone is a smooth set of co-dimension 2 in $\R^n$, the boundary regularity along this portion of the boundary fits into the framework of \cite{L88} (see also \cite{L} for an alternative exposition). We will therefore focus attention on the \textit{a priori} estimates locally around the vertex of the cone. We first construct the solutions to a pair of auxiliary problems, one to handle the source term $f$ and errors from the method of frozen coefficients, and the other to reduce to a problem with homogeneous boundary condition. Using Schauder type estimates for these auxiliary problems, we will then apply the barrier method, relying on carefully constructed barrier functions and the comparison principle for suitable problems solved by the derivatives of $u$, to show a growth condition on the gradient of the solution $u$ near the cone vertex. Finally, a standard scaling argument will convert this growth into the desired H\"older regularity. Theorem \ref{thm:perturb} is proven via a perturbative argument around the case of a continuous boundary operator (the case $s=0$).

In order to apply the barrier method to the derivatives of our solution $u$, we must find good derivatives or derivative combinations to estimate. The availability and choice of a good derivative in fact depends on the angle of $\be$. The reason for this is that when one derives an elliptic problem for the chosen derivative, the problem obtained (and, crucially, its associated boundary conditions) must be such that the comparison principle applies in order to make use of the barrier method. In particular, any zero order terms in the PDE or the oblique operator must come equipped with a good sign condition.

The outline of the paper is as follows. First, in \S\ref{sec:holder}, we give definitions and basic results for the (weighted) H\"older spaces in which we will work. With these definitions, we will be able to give more precise statements of the main theorem in the case that we have the positive result (gradient H\"older regularity up to the boundary). This is stated in Theorems \ref{thm:main} and \ref{thm:perturb}, which are proved in \S\ref{sec:mainproofs}. The proof relies on carefully constructed barrier functions, and so in \S\ref{sec:barrier} we give the construction of these barriers and relevant estimates of their directional derivatives that are used in the proof of the main result. The version of the comparison principle that we use in the main proofs is then stated and proved in \S\ref{sec:comp}. In \S\ref{sec:counter2}, we give the construction of the counterexamples of Theorem \ref{thm:counterexample}. Finally, in Appendix \ref{sec:counter1}, we provide the construction of solutions to the Neumann problem for the Laplacian which are H\"older continuous but not $C^1$ if we drop the assumption of axi-symmetry. 

\section{Weighted H\"older spaces}\label{sec:holder}
Although we will ultimately end up with a $C^{1,\al}$ estimate on the solutions $u$ of oblique derivative problems, the estimates and proofs are most conveniently stated in certain weighted H\"older spaces.

Due to the lack of regularity of the domain at the vertex of the cone, we incorporate the distance to the vertex in defining our H\"older spaces. For $k\in\N\cup\{0\}$, $\al\in(0,1]$, we define the $\sup$ norm $\|u\|_0$, standard H\"older semi-norm $[u]_\al$ and norm $\|u\|_{k,\al}$ of a function $u:\overline{\Om}\to\R$ to be
\beq
\|u\|_{0,\Om}=\sup_{{\Om}} |u|,\quad [u]_{\al,\Om}=\sup_{\substack{x_1,x_2\in{\Om},\\x_1\neq x_2}}\frac{|u(x_1)-u(x_2)|}{|x_1-x_2|^\al},\quad \|u\|_{k,\al,\Om}=\sum_{j=0}^{k}\|D^ju\|_{0,\Om}+[D^ku]_{\al,\Om},
\eeq
where $D^ju$ is the tensor of all $j$-th derivatives of $u$.

For $x,x_1,x_2\in\Om$, we define the distances $d_{x}=|x|$ and $d_{x_1,x_2}=\min\{d_{x_1},d_{x_2}\}$. For a function $u:\Om\to\R$, $k\in\N\cup\{0\}$, $\al\in(0,1]$ and $\be\in\R$, we define weighted norms 
\beqa
&\|u\|_{k,0,\Om}^{(\be)}=\sum_{j=0}^k \sup_{x\in{{\Om}}}\big(d_x^{\max\{j+\be,0\}}|D^j u(x)|\big),\\
&[u]_{k,\al,\Om}^{(\be)}=\sup_{\substack{x_1,x_2\in{{\Om}}\\x_1\neq x_2}}\Big(d_{x_1,x_2}^{\max\{k+\al+\be,0\}}\frac{|D^ku(x_1)-D^ku(x_2)|}{|x_1-x_2|^\al}\Big),\\
&\|u\|_{k,\al,\Om}^{(\be)}=\|u\|_{k,0,\Om}^{(\be)}+[u]_{k,\al,\Om}^{(\be)}.
\eeqa
We denote by $C_{k,\al}^{(\be)}(\Om)$ the space of functions whose norm $\|u\|_{k,\al,\Om}^{(\be)}$ is finite. When no confusion can arise, we usually drop the subscript $\Om$ from the definition of the norm, writing instead $\|u\|_{k,\al}^{(\be)}$ etc. 
Note in particular that if $u\in C_{2,0,\Om}^{(-1-\al)}$, then $u\in C^{1,\al}(\overline{\Om})$. The norms $\|\cdot\|_{k,\al,\om}^{(\be)}$ and $\|\cdot\|_{k,\al,\gamma}^{(\be)}$ for $\gamma\subset\d\om$ are defined similarly (with weights to the vertex).

We will need the following lemmas concerning these weighted norms. Proofs may be found in \cite[Chapter 2]{L}.
\begin{lemma}\label{lemma:holder}
(i) Suppose $\al\in(0,1]$, $\be\geq-\al$ and $\be_1,\be_2,\be_1',\be_2'\in\R$ such that $\be=\be_1+\be_2=\be_1'+\be_2'$, $\be_1,\be_2'\geq-\al$, $\be_2,\be_1'\geq0$. 
 Then for any $u\in C_{0,\al}^{(\be_1)}$, $v\in C_{0,\al}^{(\be_2')}$, we have
\beq\label{ineq:holderprod2}
[uv]_{0,\al}^{(\be)}\leq[u]_{0,\al}^{(\be_1)}\|v\|_{0}^{(\be_2)}+\|u\|_{0}^{(\be_1')}[v]_{0,\al}^{(\be_2')}.
\eeq
(ii) Suppose that $k_1,k_2\in\N\cup\{0\}$, $\al_1,\al_2\in(0,1]$, $\be_1,\be_2\in\R$ satisfy 
$$k_j+\al_j+\be_j\geq 0,\:\max\{k_j+\al_j+\be_j\}>0,\quad \text{ for }j=1,2.$$
Let $\th\in(0,1)$. Then, for $k\in\N\cup\{0\}$, $\al\in(0,1]$, $\be\in\R$ defined by
$$k+\al=\th(k_1+\al_1)+(1-\th)(k_2+\al_2),\quad \be=\th\be_1+(1-\th)\be_2,$$
there exists a constant $C>0$ such that, for any $u\in C_{k,\al}^{(\be)}$, 
\beq\label{ineq:holderinterp}
\|u\|_{k,\al}^{(\be)}\leq C\big(\|u\|_{k_1,\al_1}^{(\be_1)}\big)^\th\big(\|u\|_{k_2,\al_2}^{(\be_2)}\big)^{1-\th}.
\eeq
\end{lemma}
Finally, in the proofs of Lemmas \ref{lemma:V} and Lemma \ref{lemma:W} below, we will also need the following norms, weighted by distance to a portion of the boundary. Let $\Sigma\subset\d\Omega$ be closed and define $d_x^\Sigma=\dist(x,\Sigma)$, $d_{x_1,x_2}^\Sigma=\min\{d_{x_1}^\Sigma,d_{x_2}^\Sigma\}$. We then define
\beqa\label{def:Sigmaholder}
&\|u\|_{k,0,\Om}^{(\be),\Sigma}=\sum_{j=0}^k \sup_{x\in\Om}\big((d_x^\Sigma)^{\max\{j+\be,0\}}|D^j u(x)|\big),\\
&[u]_{k,\al,\Om}^{(\be),\Sigma}=\sup_{\substack{x_1,x_2\in\Om\\x_1\neq x_2}}\Big((d_{x_1,x_2}^\Sigma)^{\max\{k+\al+\be,0\}}\frac{|D^ku(x_1)-D^ku(x_2)|}{|x_1-x_2|^\al}\Big),\\
&\|u\|_{k,\al,\Om}^{(\be),\Sigma}=\|u\|_{k,0,\Om}^{(\be),\Sigma}+[u]_{k,\al,\Om}^{(\be),\Sigma}.
\eeqa

\section{Main estimates and proof of main theorem}\label{sec:mainproofs}
We break the proof of Theorem \ref{thm:rough} into two parts. The first concerns the case in which the oblique vector points into the first or third quadrant. In this case, we may apply the barrier technique in order to conclude that the desired H\"older regularity of the gradient holds. This is the content of Theorem \ref{thm:main} below. The second part is to prove a perturbative result for oblique vectors close to $\d_{x_n}$ ($\d_{y_1}$ in axi-symmetric coordinates). This is contained in Theorem \ref{thm:perturb} below.
\begin{thm}\label{thm:main}
Suppose $u\in C_{2,\al,\Om}^{(-1-\al)}$ is axi-symmetric and satisfies \eqref{eq:fullspace}. Let the coefficients of \eqref{eq:fullspace} satisfy 
\beq\label{ass:1}
A^{ij}\in C^0\cap C_{0,\al,\Om}^{(0)},\quad A^i,A^0\in C_{0,\al,\Om}^{(1-\al)},
\eeq
and suppose moreover that $\be$ is axi-symmetric, uniformly oblique and inward pointing on $\d\Om\setminus\{0\}$, $\be^0$ is axi-symmetric and, when considered in the $y$ coordinates,
\beq\label{ass:2}
\be,\be^0\in  C_{1,\al}^{(-\al)}(\gcone\cup\gball).\eeq
Let $\eta_1:[0,\infty)\to[0,\infty)$ be a continuous, increasing function such that $\eta_1(0)=0$ and 
\beq\label{ass:3}
|A^{ij}(x)-A^{ij}(\bar x)|\leq \eta_1(|x-\bar x|) \quad \text{ for any }x\in\Om,\,\,\bar x\in\d\Om.
\eeq
For compatibility at the intersection $\overline\Gcone\cap\overline\Gball$, we assume either
\beq\label{ass:4}
\Big|\frac{\be_{\textup{b}}}{|\be_{\textup{b}}|}\pm\frac{\be_{\textup{c}}}{|\be_{\textup{c}}|}\Big|\geq\tilde\epsilon>0\quad\text{ or } \quad \be_{\textup{b}}=\be_{\textup{c}} \text{ on }\overline\Gcone\cap\overline\Gball,
\eeq
where $\be_{\textup{b}}$ and $\be_{\textup{c}}$ are the limits of $\be$ on $\overline\Gcone\cap\overline\Gball$ from either side.\\
Finally, we assume that the data $f,g$ are axi-symmetric and
\beq\label{ass:5}
f\in C_{0,\al,\Om}^{(1-\al)},\quad g\in C_{1,\al}^{(-\al)}(\Gcone\cup\Gball).
\eeq
We write, in $y$ coordinates, 
\beq\label{eq:beta0}
\lim_{\substack{y\to 0\\ y\in\gcone}}\beta(y)=(\cos(s),\sin(s))\quad\text{ such that }s\in(-\pi+\th_0,\th_0).
\eeq Suppose $\cos(s)\sin(s)>0$. Then there exists $\al_1=\al_1(\th_0,\tilde\epsilon,s)\in(0,1)$ such that if $\al\in(0,\al_1)$ then
$$\|u\|_{2,\al}^{(-1-\al)}\leq C\big(\|f\|_{0,\al}^{(1-\al)}+\|g\|_{1,\al}^{(-\al)}+\|u\|_0\big).$$
The constant $C$ depends on the norms $\|A^{ij}\|_{0,\al}^{(0)}$, $\|A^i\|_{0,\al}^{(1-\al)}$, $\|A^0\|_{0,\al}^{(1-\al)}$, $\|\be\|_{1,\al,\gcone}^{(-\al)}$, $\|\be^0\|_{1,\al,\gcone}^{(-\al)}$, as well as $\Lambda$, $\la$, $\eta_1$, $\th_0$, $\tilde\epsilon$, $R$, $s$ and $\al$.
\end{thm}
\begin{rmk}
The constant $\al_1$ is defined in Lemma \ref{lemma:Miller}. As discussed in \S\ref{sec:holder}, this gives the estimate on the $C^{1,\al}$ norm of $u$. We remind the reader that this class of vector fields $\be$ includes oblique vector fields that are discontinuous at the vertex of the cone such as the unit normal vector field as, once we reduce to the axi-symmetric coordinates, this becomes continuous (indeed, constant) on $\gcone$. We note in passing that the assumptions of this theorem could be weakened in various directions. Firstly, the choice of boundary conditions and data on $\gball$ is unimportant for the regularity at the cone vertex, which is what we are interested in here. These boundary conditions could of course be replaced with other suitable conditions. 
\end{rmk}

We begin with two auxiliary lemmas. In order to state these lemmas, we must first define the notation we use for our frozen coefficients on the domain $\om$. Let
\beq\label{def:frozen}
L_0:=a^{ij}_0\d_{ij}+\frac{b^{2,1}_0}{y_2}\d_2,\quad \be_0=(\be_1,\be_2)=\lim_{\substack{y\to 0\\ y\in\gcone}}\beta(y),
\eeq
where $a^{ij}_0=a^{ij}(0)$ and $b^{2,1}_0=b^{2,1}(0)$. Note that $L_0$ corresponds to the operator $\overline{L}_0=A^{ij}(0)\d_{ij}$ on $\Om$ as in \eqref{eq:L0bar}.

In the following lemma, we recall the notation $\om[\rho]=\om\cap B_\rho(0)$ and define $\gsymm[\rho],\gcone[\rho]$ similarly for $\rho\in(0,R]$.
\begin{lemma}\label{lemma:V}
Suppose $g_0\in C_{1,\al}^{(-\al)}(\gcone[2\rho])$ for some $0<\rho\leq\min\{1,R/2\}$. Then there exists $V\in C_{2,\al,\om[\rho]}^{(-1-\al)}$ satisfying
\beq
\begin{cases}
a^{ij}_0\d_{ij}V=0 &\text{ in }\om[\rho],\\
\be_0\cdot DV=g_0 &\text{ on }\gcone[\rho],\\
V_{y_2}=0\ &\text{ on }\gsymm[\rho].
\end{cases}
\eeq
Moreover, for any $\de\in(0,\al]$, $V$ satisfies the estimate
\beq
\|V\|_{2,\de,\om[\rho]}^{(-1-\al)}\leq C\|g_0\|_{1,\de,\gcone[2\rho]}^{(-\al)}.
\eeq
\end{lemma}

\begin{lemma}\label{lemma:W}
Suppose $f_1\in C_{0,\al,\om}^{(1-\al)}$. There exists $W\in C_{2,\al,\om}^{(-1-\al)}$ such that 
\beq
\begin{cases}
L_0W=f_1 & \text{ in }\om,\\
W_{y_2}=0 & \text{ on }\gsymm.
\end{cases}
\eeq
Moreover, $W$ satisfies $DW(0)=0$ and, for any $\de\in(0,\al]$, the estimates 
$$|DW(y)|\leq C\|f_1\|_{0,\de}^{(1-\al)}|y|^\al,\quad |D^2W(y)|\leq C\|f_1\|_{0,\de}^{(1-\al)}|y|^{\al-1}.$$
\end{lemma}

Delaying the proofs of these lemmas temporarily, we now present the proof of Theorem \ref{thm:main}. The basic strategy of the proof is the following: we first use Lemmas \ref{lemma:V} and \ref{lemma:W} to reduce the problem to a constant coefficient problem with no source term. We then derive further problems solved by specific derivatives or derivative combinations of the solution and apply the barrier method (using the barriers of \S\ref{sec:barrier}) to get estimates of the form
$$|Du(y)|\leq C_0|y|^\al,$$
where $C_0$ depends on the data. Finally, we apply a standard scaling argument and interpolation to deduce from this the full H\"older regularity.

The precise construction of the barriers is delayed until \S\ref{sec:barrier}, as the notation we will require occurs most naturally in the proof of Theorem \ref{thm:main} below.

\begin{proof}[Proof of Theorem \ref{thm:main}]
We begin by noting that by the results of \cite{L88}, the desired estimates hold locally in  $\overline{\om}\setminus\{0\}$. Indeed, the conditions \eqref{ass:4} are precisely those required in \cite[Lemma 1.3]{L88}. It is therefore sufficient to show the estimate locally around the origin (the vertex of the cone). By a standard partition of unity argument, we may assume that $u$ is compactly supported near 0 on a ball $B_\rho(0)\cap\overline\om$, $0<\rho\leq\min\{1,R/2\}$. Moreover, without loss of generality, we may assume $u(0)=0$ (else consider $u-u(0)$) and that $Du(0)=0$ also (else consider $u-\frac{g(0)}{\be_1}y_1$). Thus also $g(0)$ may be assumed to be zero. Note that such adjustments to $u$ (subtracting an affine function) do not change the regularity of the data $f$ and $g$ due to the assumptions made on the coefficients $A^i,A^0$ and $\be,\be^0$.\\
\textbf{Step 1:} We begin by freezing coefficients. Defining 
$$f_0=f-(a^{ij}-a^{ij}_0)\d_{ij}u-b^1\d_1u-(b^2-\frac{b^{2,1}_0}{y_2})\d_2u-cu,$$
we note that on $\om[\rho]$, for any $\de\in(0,\al]$,
\beq\label{ineq:F_0}
\|f_0\|_{0,\de}^{(1-\al)}\leq \|f\|_{0,\de}^{(1-\al)}+\eta_1(\rho)\|u\|_{2,\de}^{(-1-\al)}+C[A^{ij}]_{0,\de}^{(0)}\|u\|_{2,0}^{(-1-\al)}+C(\|A^i\|_{0,\de}^{(1-\al)}+\|A^0\|_{0,\de}^{(1-\al)})\|u\|_{1,\de}^{(-1)},
\eeq
where $\eta_1(\rho)$ is the continuous, increasing function such that $\eta_1(0)=0$ from the statement of the theorem and we have used \eqref{ineq:holderprod2}.
Defining also $$g_0:=g-(\be-\be_0)\cdot Du -\be^0 u,$$ we have that
$$\be_0\cdot Du=g_0\text{ on }\gcone,$$
where, for $\de\in(0,\al]$ to be chosen later, 
\beq\label{ineq:G_0}
G_0:=\|g_0\|_{1,\de,\gcone}^{(-\al)}\leq \|g\|_{1,\de,\gcone}^{(-\al)}+\eta_2(\rho)\|u\|_{2,\de}^{(-1-\al)}+C[\be,\be^0]_{0,\de}^{(0)}\|u\|_{2,0}^{(-1-\alpha)}+C\|u\|_{1,\de}^{(-1)},
\eeq
where $\eta_2(\rho)$ is continuous, increasing, $\eta_2(0)=0$ (such an $\eta_2$ exists by \eqref{eq:beta0}) and we have used Lemma \ref{lemma:holder}(i).
Taking now the function $V$ defined by Lemma \ref{lemma:V} with boundary data $g_0$, we obtain 
\beqs
\begin{cases}
L_0(u-V)=f_0-\frac{b^{2,1}_0}{y_2}\d_2V=:f_1 & \text{ in }\om,\\
\be_0\cdot D(u-V)=0 & \text{ on }\gcone,\\
(u-V)_{y_2}=0 & \text{ on } \gsymm.
\end{cases}
\eeqs
Using the estimate of Lemma \ref{lemma:V} and \eqref{ineq:F_0}--\eqref{ineq:G_0}, we have (for $\de\in(0,\al)$ to be chosen later)
\beq\label{ineq:F_1}
F_1:=\|f_1\|_{0,\de}^{(1-\al)}\leq C\big(\|f\|_{0,\de}^{(1-\al)}+\|g\|_{1,\de,\gcone}^{(-\al)}+(\eta_1(\rho)+\eta_2(\rho))\|u\|_{2,\de}^{(-1-\al)}+\rho^{\al-\de}\|u\|_{2,0}^{(-1-\al)}+\|u\|_{1,\de}^{(-1)}\big),
\eeq
where we have used that $[A^{ij}]_{0,\de}^{(0)}+[\be,\be^0]_{0,\de}^{(0)}\leq C\rho^{\al-\de}$ on $B_\rho(0)$. Note also that $g_0(0)=0$, so as $DV$ is continuous up to the origin, we must have $DV(0)=0$.

Next we apply Lemma \ref{lemma:W} with source term $f_1$ to obtain a function $W$ satisfying
\beqs
\begin{cases}
L_0W=f_1 & \text{ in }\om,\\
W_{y_2}=0 & \text{ on }\gsymm,
\end{cases}
\eeqs
and the estimates
$$|DW(y)|\leq CF_1|y|^{\al},\quad |D^2W(y)|\leq CF_1|y|^{\al-1}.$$
\textbf{Step 2:} We now proceed to derive suitable problems for the derivatives of $u$ and apply the barrier method to obtain a growth rate estimate on $|Du|$ near the vertex.\\
We write $\nu=(\nu_1,\nu_2)$ for the inward unit normal on $\gcone$ and $\be_0=(\beta_1,\beta_2)$, where $\beta_1,\beta_2$ have the same sign (by the assumption made on $s$).
Define
\beq
v_1=(u-V)_{y_1},\quad v_2=(u-V)_{y_2}.
\eeq
Define coordinates $(z_1,z_2)$ such that $\d_{z_1}$ is parallel to $\beta_0=(\beta_1,\beta_2)$ and $\d_{z_2}$ is parallel to $\tau$, the tangent to $\gcone$, so
\beq\begin{pmatrix}\label{def:zcoords}
z_1\\z_2
\end{pmatrix}=\begin{pmatrix}
\beta_1 & \beta_2\\
\nu_2 & -\nu_1
\end{pmatrix}\begin{pmatrix}
y_1\\y_2
\end{pmatrix}.\eeq
The reverse coordinate change is given by
$$\begin{pmatrix}
y_1\\y_2
\end{pmatrix}=\frac{1}{\nu\cdot\beta_0}\begin{pmatrix}
\nu_1 & \beta_2\\
\nu_2 & -\beta_1
\end{pmatrix}\begin{pmatrix}
z_1\\z_2
\end{pmatrix}.$$
From here on, we write  $\nu\cdot\beta_0=\epsilon>0$ by obliqueness (recall $\be$, $\nu$ are both inward pointing).

 Now by changing coordinates in the operator $L_0$, we find coefficients $\tilde{a}^{ij}$ such that for any $\psi:\om\to\R$,
 \beq\label{def:tildecoords}
 \tilde{a}^{ij}\psi_{z_iz_j}=a^{ij}_0\psi_{y_iy_j}\quad\text{ and }\quad\frac{\la}{\epsilon^2}\leq \tilde{a}^{11},\tilde{a}^{22}\leq \frac{\La}{\epsilon^2}.
 \eeq
We derive an oblique derivative condition for $v_1$ by computing on $\gcone$. Noting that $(u-V)_{z_1z_2}=0$ on $\gcone$, we calculate on $\gcone$
\beqas
\d_{z_1}v_1=&\,\frac{1}{\epsilon}\d_{z_1}\big(\nu_1(u-V)_{z_1}+\beta_2(u-V)_{z_2}\big)=\frac{1}{\epsilon}\nu_1(u-V)_{z_1z_1}\\
=&\,\frac{1}{\epsilon}\frac{\nu_1}{\tilde{a}^{11}} f_1-\frac{1}{\epsilon}\nu_1\frac{\tilde{a}^{22}}{\tilde a^{11}}(u-V)_{z_2z_2}-\frac{1}{\epsilon}\frac{\nu_1}{\tilde a^{11}}\frac{b_0^{2,1}}{y_2}(u-V)_{y_2}\\
=&\,\frac{1}{\epsilon}\frac{\nu_1}{\tilde{a}^{11}} f_1-\frac{1}{\epsilon}\frac{\nu_1}{\beta_2}\frac{\tilde{a}^{22}}{\tilde a^{11}}\big(\nu_1(u-V)_{z_1z_2}+\beta_2(u-V)_{z_2z_2}\big)+\frac{1}{\epsilon}\frac{\nu_1}{\tilde a^{11}}\frac{\beta_1}{\beta_2}\frac{b_0^{2,1}}{y_2}(u-V)_{y_1}\\
=&\,\frac{1}{\epsilon}\frac{\nu_1}{\tilde{a}^{11}} f_1-\frac{\nu_1}{\beta_2}\frac{\tilde a^{22}}{\tilde a^{11}}\d_{z_2}v_1+\frac{1}{\epsilon}\frac{\nu_1}{\tilde a^{11}}\frac{\beta_1}{\beta_2}\frac{b_0^{2,1}}{y_2}v_1,
\eeqas
where we have also used that on $\gcone$ $$(u-V)_{y_2}=-\frac{\beta_1}{\beta_2}(u-V)_{y_1}\text{ as }\be_0\cdot D(u-V)=0.$$
Thus we obtain the boundary condition
\beq
 M_1v_1:=\be_0\cdot Dv_1+\frac{\nu_1}{\beta_2}\frac{\tilde a^{22}}{\tilde a^{11}}\tau\cdot Dv_1-\frac{1}{\epsilon}\frac{\nu_1}{\tilde a^{11}}\frac{\beta_1}{\beta_2}\frac{b_0^{2,1}}{y_2}v_1=\frac{1}{\epsilon}\frac{\nu_1}{\tilde{a}^{11}} f_1.
\eeq
Noting that $\nu_1,\beta_1/\beta_2,\tilde{a}^{11},b_0^{2,1},\epsilon>0$, the zero order term in this boundary operator comes equipped with a negative sign, which is necessary for the application of the comparison principle. 

As $\d_{y_1}$ commutes with $L_0$ and the Neumann condition on $\gsymm$, we arrive at the following problem for $v_1$,
\beq
\begin{cases}
L_0(v_1-W_{y_1})=0 &\text{ in }\om,\\
(v_1-W_{y_1})_{y_2}=0 &\text{ on }\gsymm,\\
 M_1(v_1-W_{y_1})= \frac{1}{\epsilon}\frac{\nu_1}{\tilde{a}^{11}} f_1-M_1W_{y_1} &\text{ on }\gcone.
\end{cases}
\eeq
Let $v_\al$ be the Miller barrier function as in \S\ref{sec:barrier} and choose $\al_1$ as in Lemma \ref{lemma:Miller} so that for $\al\in(0,\al_1)$, we have the boundary inequality
$$M_1v_\al\leq -c_1|y|^{\al-1}\text{ on }\gcone.$$
Therefore, using the estimate of Lemma \ref{lemma:W} for $W$, we may choose a constant $\widehat{C}=C_1\big(F_1+\|u\|_1\big)$, where $C_1>0$ is independent of $u$ and $F_1$, such that $\hat{v}=\widehat{C}v_\al$ satisfies
\beq
\begin{cases}
L_0(\hat{v}\pm(v_1-W_{y_1}))\leq 0 &\text{ in }\om,\\
(\hat{v}\pm(v_1-W_{y_1}))_{y_2}=0 &\text{ on }\gsymm,\\
 M_1(\hat{v}\pm(v_1-W_{y_1}))\leq 0 &\text{ on }\gcone,\\
\hat{v}\pm(v_1-W_{y_1})\geq 0 &\text{ on }\gball,
\end{cases}
\eeq
where we have used that $$\big|\frac{1}{\epsilon}\frac{\nu_1}{\tilde{a}^{11}} f_1-M_1W_{y_1}\big|\leq C(|f_1|+|D^2W|+\frac{1}{y_2}|DW|)\leq CF_1|y|^{\al-1} \text{ on }\gcone.$$
As $D(u-V-W)(0)=0$, we also have the one-point Dirichlet condition $\hat{v}\pm(v_1-W_{y_1})(0)=0$. 
Applying the first version of the comparison principle in Theorem \ref{thm:comparison}, we obtain that
$|v_1-W_{y_1}|\leq |\hat{v}|$, and hence, applying also the estimate for $V$, 
\beq\label{ineq:v1est}
|u_{y_1}(y)|\leq C\big(F_1+G_0+\|u\|_1\big)|y|^{\al}.
\eeq
Next, we make a similar argument for a second derivative direction $w=v_1+\eps v_2$, where $\eps>0$ is sufficiently small. First, we derive a PDE for $w$ in the domain $\om$: Define $\tilde W=W_{y_1}+\eps W_{y_2}$. Then
$$0=L_0(w-\tilde W)-\eps\frac{b_0^{2,1}}{y_2^2}(u-V-W)_{y_2}=L_0(w-\tilde W)-\frac{b_0^{2,1}}{y_2^2}(w-\tilde W)+\frac{b_0^{2,1}}{y_2^2}(u-V-W)_{y_1},$$
hence 
\beq
L_0(w-\tilde W)-\frac{b_0^{2,1}}{y_2^2}(w-\tilde W)=-\frac{b_0^{2,1}}{y_2^2}(u-V-W)_{y_1}.
\eeq
Note that the zero order term on the left comes equipped with the correct sign for application of the comparison principle as $b^{2,1}_0>0$. Moreover, by Lemmas \ref{lemma:V} and \ref{lemma:W} and \eqref{ineq:v1est}, we already have an estimate for the right hand side:
\beq\label{ineq:pdeest}
\big|\frac{b_0^{2,1}}{y_2^2}(u-V-W)_{y_1}\big|\leq C\big(F_1+G_0+\|u\|_1\big)\frac{|y|^\al}{y_2^2}.\eeq
Next, we find an oblique derivative condition on $\gcone$ by using the PDE for $u-V$ in $z$ coordinates and the boundary condition $(u-V)_{z_1z_2}=0$ on $\gcone$:
\beqas
\d_{z_1}w=&\,\frac{1}{\epsilon}\d_{z_1}\big((\nu_1+\eps \nu_2)(u-V)_{z_1}+(\beta_2-\eps \beta_1)(u-V)_{z_2}\big)=\frac{1}{\epsilon}(\nu_1+\eps \nu_2)(u-V)_{z_1z_1}\\
=&\,\frac{1}{\epsilon}\frac{\nu_1+\eps \nu_2}{\tilde{a}^{11}} f_1-\frac{1}{\epsilon}(\nu_1+\eps \nu_2)\frac{\tilde{a}^{22}}{\tilde a^{11}}(u-V)_{z_2z_2}-\frac{1}{\epsilon}\frac{\nu_1+\eps \nu_2}{\tilde a^{11}}\frac{b_0^{2,1}}{y_2}(u-V)_{y_2}\\
=&\,\frac{1}{\epsilon}\frac{\nu_1+\eps \nu_2}{\tilde{a}^{11}} f_1-\frac{1}{\epsilon}\frac{\nu_1+\eps \nu_2}{\beta_2-\eps \beta_1}\frac{\tilde{a}^{22}}{\tilde a^{11}}\big((\nu_1+\eps \nu_2)(u-V)_{z_1z_2}+(\beta_2-\eps \beta_1)(u-V)_{z_2z_2}\big)\\
&+\frac{1}{\epsilon}\frac{\nu_1+\eps \nu_2}{\tilde a^{11}}\frac{\beta_1}{\beta_2-\eps \beta_1}\frac{b_0^{2,1}}{y_2}w\\
=&\,\frac{1}{\epsilon}\frac{\nu_1+\eps \nu_2}{\tilde{a}^{11}} f_1-\frac{\nu_1+\eps \nu_2}{\beta_2-\eps \beta_1}\frac{\tilde a^{22}}{\tilde a^{11}}\d_{z_2}w+\frac{1}{\epsilon}\frac{\nu_1+\eps \nu_2}{\tilde a^{11}}\frac{\beta_1}{\beta_2-\eps \beta_1}\frac{b_0^{2,1}}{y_2}w,
\eeqas
where we have also used that $w=-\frac{\be_2-\eps\be_1}{\be_1}(u-V)_{y_2}$ on $\gcone$ from the boundary condition.
Thus we arrive at the problem satisfied by $w$:
\beq
\begin{cases}
L_0(w-\tilde W)-\frac{b_0^{2,1}}{y_2^2}(w-\tilde W)=-\frac{b_0^{2,1}}{y_2^2}(u-V-W)_{y_1} &\text{ in }\om,\\
w-\tilde W=u_{y_1}-V_{y_1}-W_{y_1} &\text{ on }\gsymm,\\
 M_2(w-\tilde W)=\frac{1}{\epsilon}\frac{\nu_1+\eps \nu_2}{\tilde{a}^{11}} f_1- M_2\tilde W &\text{ on }\gcone,
\end{cases}
\eeq
where $M_2$ is the operator acting on functions $\psi$ by $$ M_2\psi=\be_0\cdot D\psi+\frac{\nu_1+\eps \nu_2}{\beta_2-\eps \beta_1}\frac{\tilde a^{22}}{\tilde a^{11}}\tau\cdot D\psi-\frac{1}{\epsilon}\frac{\nu_1+\eps \nu_2}{\tilde a^{11}}\frac{\beta_1}{\beta_2-\eps \beta_1}\frac{b_0^{2,1}}{y_2}\psi.$$
Taking $\eps>0$ sufficiently small so that $\nu_1+\eps\nu_2>0$ (recall as $\nu$ is inward pointing, $\nu_1=\sin\th_0>0$) and $\textup{sgn}(\beta_2-\eps \beta_1)=\textup{sgn}(\be_2)$ (recall $\be_2\neq 0$ by assumption), we have that the zero order term in $M_2$ has a negative coefficient. If $\eps>0$ is sufficiently small, we also have the estimate of Lemma \ref{lemma:Miller} for the Miller barrier. Therefore we may take $\bar{v}=\overline{C}v_\al$ with $\overline{C}=C_2\big(F_1+G_0+\|u\|_1\big)$ with $C_2>0$. Using the  positivity of $v_\al$, we have
$$L_0\bar v-\frac{b_0^{2,1}}{y_2^2}\bar v\leq -\overline{C}c_*\frac{b_0^{2,1}}{y_2^2}|y|^{\al},
$$
where $c_*>0$ is as in \eqref{ineq:miller} and $b_0^{2,1}>0$.
 Thus from \eqref{ineq:pdeest}, for $C_2$ sufficiently large, we have
 $$L_0\big(\bar v\pm(w-\tilde{W})\big)-\frac{b_0^{2,1}}{y_2^2}\big(\bar v\pm(w-\tilde{W})\big)\leq 0.$$
 Similarly, on $\gsymm$, we use the Neumann condition $u_{y_2},V_{y_2},W_{y_2}=0$ to make the estimate
 $$\bar v\pm(w-\tilde{W})\geq \bar v-\big|u_{y_1}-V_{y_1}-W_{y_1}\big|\geq 0,$$
 where we have used the already obtained estimate $|u_{y_1}(y)|\leq C\big(F_1+G_0+\|u\|_1\big)|y|^{\al}$ and the estimates on $DV$, $DW$.
 
  We therefore apply the second form of the comparison principle in Theorem \ref{thm:comparison} to $\bar{v}\pm(w-\tilde{W})$. This gives
$$|u_{y_1}+\eps u_{y_2}|\leq C\big(F_1+G_0+\|u\|_1\big)|y|^{\al},$$
and hence also
$$|Du|\leq C\big(F_1+G_0+\|u\|_1\big)|y|^{\al}.$$
\textbf{Step 3:} We now conclude via a standard scaling argument. For the convenience of the reader, we include the argument here. As will become clear, such an argument could be performed either on $\om$ or on the original domain $\Om$. For ease of notation, we continue to work on $\om$.

As we have assumed without loss of generality that $u(0)=0$, we have obtained the inequality
\beq\label{ineq:scaling}
|u(y)|\leq C_0|y|^{\al+1},
\eeq
where $C_0=C\big(F_1+G_0+\|u\|_1\big)$. 

Let $y\in\overline{\om}[R/2]$ (without loss of generality, we suppose here $R\geq 1$), $y\neq 0$, and recall the notation $d_y=|y|$. Then at least one of the following is true:
\begin{itemize}
\item[(i)] $B_{\frac{d_y}{10 }}(y)\subset\om$,
\item[(ii)] $y\in B_{\frac{d_{\hat{y}}}{2}}(\hat{y})$ for some $\hat{y}\in\gcone$,
\item[(iii)] $y\in B_{\frac{d_{\hat{y}}}{2}}(\hat{y})$ for some $\hat{y}\in\gsymm$.
\end{itemize}
We focus on case (ii) here, as cases (i) and (iii) may be treated similarly (for case (iii) we return to $\Omega$ and note that $\gsymm$ lies in the interior of $\Omega$). We therefore assume we have a point $\hat{y}\in\gcone$ such that $\hat{d}=\frac{1}{2}d_{\hat{y}}\in(0,1)$ and derive an estimate on $B_{\frac{\hat{d}}{2}}(\hat{y})$.

Define new coordinates $z=\frac{y-\hat{y}}{\hat{d}}$. Rescaling $\om\cap B_{r\hat{d}}(\hat{y})$ for $r\in(0,1]$ gives us the new domain
$$\om_r^{\hat{y}}:=\begin{cases}
B_r(0)\cap\{z_2<\tan\th_0 z_1\} &\text{ if }\th_0\in(0,\frac{\pi}{2}),\\
B_r(0)\cap\{z_2>\tan\th_0 z_1\} &\text{ if }\th_0\in(\frac{\pi}{2},\pi).
\end{cases}$$
On $\om_1^{\hat{y}}$, we define a new unknown $$v(z)=\frac{u(\hat{y}+\hat{d}z)}{\hat{d}^{1+\al}}.$$
Then from the inequality \eqref{ineq:scaling}, we have the estimate
$$\|v\|_{0,\om_1^{\hat{y}}}= \|u\|_{0,\om\cap B_{\hat{d}}(\hat{y})}\hat{d}^{-1-\al}\leq CC_0.$$
Defining also
$$\hat{f}(z)=\hat{d}^{1-\al}f(\hat{y}+\hat{d}z),\quad \hat{g}(z)=\hat{d}^{-\al}g(\hat{y}+\hat{d}z),$$
one easily sees that $v$ satisfies the equation
\beq
\begin{cases}
a^{ij}\d_{ij}v+\hat{d}b^i\d_i v+\hat{d}^2 cv=\hat{f} & \text{ in } \om_1^{\hat{y}},\\
\be\cdot Dv+\be^0 v=\hat{g} &\text{ on }\om_{1,\textrm{cone}}^{\hat{y}}=\om_1^{\hat{y}}\cap\{z_2=\tan\th_0 z_1\}, 
\end{cases}
\eeq
where the coefficients $a^{ij},b^i,c,\be,\be^0$ are all evaluated at $y(z)=\hat{y}+\hat{d}z$. It is then straightforward to obtain the estimates
$$\|\hat{f}\|_{0,\al,\om_1^{\hat{y}}}\leq C\|f\|_{0,\al,\om}^{(1-\al)},\quad \|\hat{g}\|_{1,\al,\om_{1,\textrm{cone}}^{\hat{y}}}\leq C\|g\|_{1,\al,\gcone}^{(-\al)}.$$
Thus the standard elliptic regularity theory for oblique problems on smooth domains, e.g.~in \cite[Theorem 6.30]{GT} (note that $\hat{d}b^i$ is a bounded, regular coefficient on $\om_{1}^{\hat{y}}$), gives the estimate
$$\|v\|_{2,\al,\om_{1/2}^{\hat{y}}}\leq C\big(\|v\|_{0,\om_{1}^{\hat{y}}}+\|\hat{f}\|_{0,\al,\om_1^{\hat{y}}}+\|\hat{g}\|_{1,\al,\om_{1,\textrm{cone}}^{\hat{y}}}\big)\leq C\big(C_0+\|f\|_{0,\al,\om}^{(1-\al)}+\|g\|_{1,\al,\gcone}^{(-\al)}\big).$$
It is simple to check the equivalence $$\|v\|_{2,\al,\om_{1/2}^{\hat{y}}}\approx \|u\|_{2,\al,\om\cap B_{\hat{d}/2}(\hat{y})}^{(-1-\al)},$$ and so we use this estimate and similar estimates for interior balls (including balls centred on $\gsymm$ as discussed above) to cover the domain and apply \eqref{ineq:G_0}--\eqref{ineq:F_1} to
 arrive (via an argument such as that of \cite[Theorem 4.8]{GT}) at the estimate
\beqa
\|u\|_{2,\al}^{(-1-\al)}\leq&\, C\big(F_1+G_0+\|f\|_{0,\al}^{(1-\al)}+\|g\|_{1,\al}^{(-\al)}\big)\\
\leq&\, C\big(\|f\|_{0,\al}^{(1-\al)}+\|g\|_{1,\al}^{(-\al)}+\eta(\rho)\|u\|_{2,\al}^{(-1-\al)}+\|u\|_{1,\de}^{(-1)}+\|u\|_0\big),
\eeqa
where $\eta(\rho)=\eta_1(\rho)+\eta_2(\rho)+\rho^{\al-\de}$.
Choosing $\de>0$ such that $\de<\al$ and $\de\leq(1+\al)^{-1}$, we use the interpolation estimate for weighted H\"older spaces, \eqref{ineq:holderinterp}, to observe that
$$\|u\|_{1,\de}^{(-1)}\leq C\big(\|u\|_{2,\al}^{(-1-\al)}\big)^\th\big(\|u\|_0\big)^{1-\th}$$
for some $\th\in(0,1)$. Thus, by Young's inequality, for any $\eps>0$, we have 
$$\|u\|_{2,\al}^{(-1-\al)}\leq C\big(\|f\|_{0,\al}^{(1-\al)}+(\eta(\rho)+\eps)\|u\|_{2,\al}^{(-1-\al)}+\|g\|_{1,\al}^{(-\al)}+\|u\|_0\big).$$
Taking now $\rho,\eps>0$ sufficiently small so that $C(\eta(\rho)+\eps) <\frac{1}{2}$, we conclude the proof.
\end{proof}
\begin{proof}[Proof of Lemma \ref{lemma:V}]
We find $V$ by applying \cite[Theorem 1.4]{L88} to the following problem on all of $\om$:
\beq\label{eq:V}
\begin{cases}
a^{ij}_0\d_{ij}V=0 &\text{ in }\om,\\
\bar\be_0\cdot DV+\bar\be^0 V=\bar g &\text{ on }\gcone\cup\gball,\\
V_{y_2}=0 &\text{ on }\gsymm,
\end{cases}
\eeq
where $\bar g$ and $\bar\be_0$ are H\"older continuous extensions of $g_0$ and $\be_0$ to all of $\gcone\cup\gball$ such that $\|\bar g\|_{1,\de}^{(-\al)}\leq C\| g_0\|_{1,\de}^{(-\al)}$ and $\bar\be_0$ remains uniformly oblique with a similar estimate. $\bar\be^0$ is a smooth scalar function such that $\bar\be^0\leq 0$, $\bar\be^0\not\equiv 0$, and $\bar\be^0=0$ on $\gcone[\rho]$. Then Theorem 1.4 of \cite{L88} gives the estimate
$$\|V\|_{2,\de}^{(-1-\al),\d\om}\leq C\|\bar g\|_{1,\de}^{(-\al)}\leq C\| g_0\|_{1,\de}^{(-\al)},$$
where the H\"older spaces with weight up to the boundary were defined in \eqref{def:Sigmaholder}.
To remove the dependence on the full boundary $\d\om$ in the norm on the left, we use standard regularity theory away from the vertex and observe that the arguments of \cite{L88} apply equally when we only weight the H\"older spaces with distance to the vertex (as the source term in the PDE of \eqref{eq:V} is zero, and hence is not singular along $\d\om$). This gives us the desired estimate,
$$\|V\|_{2,\de}^{(-1-\al)}\leq C\| g_0\|_{1,\de}^{(-\al)}.$$
\end{proof}

\begin{proof}[Proof of Lemma \ref{lemma:W}]
We define $W$ by solving a further auxiliary problem. We begin by observing that, by construction and assumption \eqref{eq:L0bar} on $A^{ij}(0)$, $L_0$ extends to the (axi-symmetric) operator $\overline{L}_0$ on $\Om$ defined by
$$\overline{L}_0w=A^{ij}(0)\d_{ij}w \text{ for $w:\Om\to\R$}.$$
We now define a smooth vector field $\tilde\be$ such that $\tilde\be=e_n$ on $B_{R/2}(0)\cap\d\Om$ and $\tilde\be$ is uniformly oblique on all of $\d\Om$ and axi-symmetric (here $e_n$ is the standard Cartesian unit vector in the $x_n$ direction). Choose a smooth function $\tilde\be^0\leq 0$ on $B_{2R}(0)$ such that $\tilde\be^0=0$ on $B_{R/2}$ and $\tilde\be^0\not\equiv0$. We then define $W$ as the solution to the problem 
\beq
\begin{cases}
\overline{L}_0 W=f_1 &\text{ in }\Om,\\
\tilde\be\cdot DW- \tilde\be^0W=0 &\text{ on }\d\Om.
\end{cases}
\eeq
Such a solution exists by \cite[Corollary 3.3]{L87} as the oblique vector $\tilde\be$ is continuous and the zero order term in the boundary condition is negative. Moreover, that same theorem gives the estimate
$$\|W\|_{2,\de,\Om}^{(-1-\al),\d\Om}\leq C\|f_1\|_{0,\de,\Om}^{(1-\al),\d\Om},$$
where we recall the notation for H\"older spaces weighted by distance to the boundary as in \eqref{def:Sigmaholder}.
As in the previous proof, we note that the methods of \cite{L87} (see especially Lemmas 2.1--2.2 and also the proof of \cite[Lemma 1.2]{L88}) improve this estimate to
\beq\label{eq:Wnorm}
\|W\|_{2,\de,\Om}^{(-1-\al)}\leq C\|f_1\|_{0,\de,\Om}^{(1-\al)}.
\eeq
It only remains to show that $W$, in the $y$ coordinates, satisfies also the Neumann condition $W_{y_2}=0$ on $\gsymm$ and deduce the growth conditions on $DW$, $D^2W$. From the uniqueness part of \cite[Theorem 3.2]{L87}, as the operator $\overline{L}_0$ is invariant under axial rotations by \eqref{eq:L0bar} and the data and boundary condition are both axi-symmetric, $W$ is also axi-symmetric, so we may return to $\om$ and deduce $W_{y_2}=0$ on $\gsymm$. Finally, combining the two boundary conditions at the origin, we deduce that $DW(0)=0$, and therefore obtain the growth rates $|DW|\leq C\|f_1\|_{0,\de,\Om}^{(1-\al)}|y|^\al$, $|D^2W|\leq C\|f_1\|_{0,\de,\Om}^{(1-\al)}|y|^{\al-1}$ from \eqref{eq:Wnorm} (where we have recalled that $\|W\|_{1,\al}\leq C\|W\|_{2,\de}^{(-1-\al)}$ for the former estimate).
\end{proof}
The next theorem covers the remaining part of Theorem \ref{thm:rough}: the case $s\in(s_1, 0]$. As stated in the introduction, the theorem is a perturbative result around the case $s=0$, where the boundary operator is continuous.
\begin{thm}\label{thm:perturb}
Let $\th_0\in(0,\pi)$ and $u\in C_{2,\al}^{(-1-\al)}$ be an axi-symmetric solution of \eqref{eq:fullspace}. We assume \eqref{ass:1}--\eqref{ass:5} hold. Then there exists $\al_0=\al_0(\th_0,\tilde\epsilon,\La/\la)\in(0,1)$ such that if $\al\in(0,\al_0)$, then there exists $s_1=s_1(\th_0,\al,\La/\la)\in(0,\frac{\pi}{2})$ such that if 
\beq\label{eq:beta02}
\be_0=\lim_{\substack{y\to 0\\ y\in\gcone}}\beta(y)=(\cos(s),\sin(s))\quad\text{ such that }s\in(-s_1,s_1),
\eeq then
$$\|u\|_{2,\al}^{(-1-\al)}\leq C\big(\|f\|_{0,\al}^{(1-\al)}+\|g\|_{1,\al}^{(-\al)}+\|u\|_0\big).$$
The constant $C$ depends on the norms $\|A^{ij}\|_{0,\al}^{(0)}$, $\|A^i\|_{0,\al}^{(1-\al)}$, $\|A^0\|_{0,\al}^{(1-\al)}$, $\|\be\|_{1,\al,\gcone}^{(-\al)}$, $\|\be^0\|_{1,\al,\gcone}^{(-\al)}$, as well as $\Lambda$, $\la$, $\eta_1$, $\th_0$, $\tilde\epsilon$, $R$, $s_1$ and $\al$.

If $[\th_1,\th_2]\subset(0,\pi)$, then $\al_0$ may be uniform with respect to $\th_0\in[\th_1,\th_2]$ and then $s_1$ and $C$ may also be taken uniform with respect to $\th_0$ and $\alpha\in(0,\al_0)$ .
\end{thm}

\begin{proof}
We perturb around the case of a continuous oblique vector, $s=0$ by recalling the results of Lieberman \cite[Proposition 3.1]{L87} (alternatively see \cite[Section 4.1]{L}). In particular, by working on the symmetry domain $\om$, we may define a new boundary vector
$$\tilde\be=\be-(0,\sin(s)),$$
so that $u$ satisfies
\beq
\begin{cases}
a^{ij}\d_{ij}u+b^i\d_iu+cu=f &\text{ in }\om,\\
\tilde\be \cdot Du+\be^0u=g-\sin(s)\d_{y_2}u &\text{ on }\gcone\cup\gball,\\
u_{y_2}=0 &\text{ on }\gsymm.
\end{cases}
\eeq
From \cite[Proposition 3.1]{L87} (applying the same argument as in the proof of Lemma \ref{lemma:W} above to weight only by distance to the vertex, not all of $\d\om$), we then have the estimate
\beq\label{ineq:absorb}
\|u\|_{2,\al,\om}^{(-1-\al)}\leq C\big(\|f\|_{0,\al}^{(1-\al)}+\|g\|_{1,\al}^{(-\al)}+\|\sin(s)\d_{y_2}u\|_{1,\al}^{(-\al)}+\|u\|_0\big),
\eeq
where $C>0$ depends on the H\"older norms of the coefficients and the ellipticity. 
Thus if $s\in(-s_1,s_1)$ and $s_1$ is sufficiently small, then we make the further estimate
$$C\|\sin(s)u_{y_2}\|_{1,\al}^{(-\al)}\leq Cs_1\|u\|_{2,\al}^{(-1-\al)}\leq \frac 12\|u\|_{2,\al}^{(-1-\al)}$$
for $s_1$ sufficiently small and conclude the claimed estimate by absorbing this term onto the left in \eqref{ineq:absorb} and returning to the original domain $\Om$.

To show that the estimates may be taken to be locally uniform with respect to $\th_0$, a careful inspection of the proofs of \cite[Lemma 2.1, Proposition 3.1]{L87} (see also \cite[Section 4.5]{L}) shows that the dependence of $\al_0$ and the constant $C>0$ on $\th_0$ comes solely from the estimates of the Miller barrier, as in \S\ref{sec:barrier} and Lemma \ref{lemma:Miller} below. However, it is clear from the construction in \cite{Miller} (compare \cite[Chapter 3]{L} and also Remark \ref{rmk:barrier} below) that the barrier function $v_\al$ and the constant $\al_0(\th_0)$ depend continuously on $\th_0$. Thus, given $[\th_1,\th_2]\subset(0,\pi)$, we may take $\al_*<\min_{\th_0\in[\th_1,\th_2]}\{\al_0(\th_0)\}$ and use a single barrier $v_\al$ with $\al\leq\al_*$ to show that the constants are all locally uniform with respect to $\th_0$. The uniform dependence of $s_1$ on $\al\leq \al_*$ then follows directly from the proof above as the constant $C>0$ may now be taken to be uniform.
\end{proof}

\section{Barrier functions for oblique problems on cones}\label{sec:barrier}
A vital tool in the proof of Theorem \ref{thm:main} is the \textit{Miller barrier}. We recall from \cite{Miller} (see also \cite[Chapter 3]{L}) that for each $\al\in(0,1)$ and $\La>\la>0$, there exists a function $F_\al(\th)$ such that 
$$v_\al=r^\al F_\al(\th)$$
satisfies, for all constant coefficient operators of the form $A_0^{ij}$ with ellipticity $\la$ and upper bound $\La$,
$$A^{ij}_0v_\al\leq 0\text{ in }\Om,$$
and that, moreover, $F_\al'(0)=0$, so that $v_\al$ is a well-defined axi-symmetric function on $\Om$. By construction (see \cite[Lemmas 3.5--3.8]{L}), the limits 
\beq
\lim_{\al\to0+}F_\al(\th)=1,\quad \lim_{\al\to0+} F_\al'(\th)=0
\eeq
hold uniformly for $\th\in(0,\pi)$.
For each $\th_0\in(0,\pi)$, there exists an $\al_0(\th_0)\in(0,1]$ such that for each $\al\in(0,\al_0)$, $c_*\leq F_\al(\th)\leq 1$ for all $\th\in[0,\th_0]$ and some $c_*>0$ (depending on $\al$). In addition, $F_\al'(\th)<0$ on $(0,\th_0]$. For $\al\in(0,\al_0)$, we refer to the function $v_\al$ as the \textit{Miller barrier} and note that, by construction,
\beq\label{ineq:miller}
c_*|y|^\al\leq v_\al\leq |y|^\al.
\eeq
To use the Miller barrier in proving Theorem \ref{thm:main}, we need to investigate its behaviour under certain oblique boundary operators. Converting to axi-symmetric coordinates, we first need some notation. Let $\be_0=(\beta_1,\beta_2)$ be a constant, inward pointing, oblique vector and suppose that $A^{ij}_0$ is a constant matrix of ellipticity $\la$ and upper bound $\La$. Define coordinates $(z_1,z_2)$ on $\om$ as in \eqref{def:zcoords} such that $\d_{z_1}=\d_\beta$, $\d_{z_2}=\d_\tau$, where $\tau=\nu^\perp$ is the unit tangent to $\gcone$ and $\nu$ is the inward unit normal. In these coordinates, following \eqref{def:tildecoords}, the elliptic operator takes the form
$$A^{ij}_0\d_{x_ix_j}\psi=\tilde{a}^{ij}\d_{z_iz_j}\psi+\frac{b_0^{2,1}}{y_2}\d_{y_2}\psi\text{ for axi-symmetric functions }\psi,$$
where the constant coefficients $\tilde{a}^{ij}$ satisfy the ellipticity
$$\frac{\la}{\epsilon^2}\leq \tilde{a}^{11},\tilde{a}^{22}\leq\frac{\La}{\epsilon^2},\quad \epsilon=\be_0\cdot\nu>0.$$
With this notation, we obtain the following lemma.
\begin{lemma}\label{lemma:Miller}
Let $\be_0=(\beta_1,\beta_2)$ be a constant, inward pointing, oblique vector on $\gcone$ such that $\beta_1,\beta_2\neq0$ have the same sign, $\th_0\in(0,\pi)$. Suppose that the constant coefficients $\tilde{a}^{ij}$ are derived from a constant matrix $A^{ij}_0$ of ellipticity $\la$ and upper bound $\La$ as above. Then there exists $\al_1\in(0, \al_0]$ such that for all $\al\in(0,\al_1)$, there exists $c_1(\al)>0$ so that the Miller barrier $v_\al$ satisfies
$$ M_1v_\al:=\be_0\cdot Dv_\al+\frac{\nu_1}{\beta_2}\frac{\tilde a^{22}}{\tilde a^{11}}\tau\cdot Dv_\al-\frac{1}{\epsilon}\frac{\nu_1}{\tilde a^{11}}\frac{\beta_1}{\beta_2}\frac{b_0^{2,1}}{y_2}v_\al\leq -c_1|y|^{\al-1} \text{ on }\gcone,$$
and also, for $\eps>0$ sufficiently small (depending on $\th_0$, $\beta_1$, $\beta_2$, $\al$, $\La$, $\la$), 
$$ M_2v_\al:=\be_0\cdot Dv_\al+\frac{\nu_1+\eps \nu_2}{\beta_2-\eps \beta_1}\frac{\tilde a^{22}}{\tilde a^{11}} \tau\cdot  Dv_\al-\frac{1}{\epsilon}\frac{\nu_1+\eps \nu_2}{\tilde a^{11}}\frac{\beta_1}{\beta_2-\eps \beta_1}\frac{b_0^{2,1}}{y_2}v_\al\leq -c_1|y|^{\al-1} \text{ on }\gcone.$$
\end{lemma}

\begin{proof}
Note first that
\beqas
(v_\al)_{y_1}=\al r^{\al-1}\cos\th F_\al(\th)-r^{\al-1}\sin\th F_\al'(\th),\quad (v_\al)_{y_2}=\al r^{\al-1}\sin\th F_\al(\th)+r^{\al-1}\cos\th F_\al'(\th).
\eeqas
Then, noting $\nu=(\sin\th_0,-\cos\th_0)$,
\beqas
M_1v_\al=&\,\be_0\cdot Dv_\al+\frac{\nu_1}{\beta_2}\frac{\tilde a^{22}}{\tilde a^{11}}\tau\cdot Dv_\al-\frac{1}{\epsilon}\frac{\nu_1}{\tilde a^{11}}\frac{\beta_1}{\beta_2}\frac{b_0^{2,1}}{y_2}v_\al\\
=&\,r^{\al-1}\Big(F_\al(\th_0)\big(\al \beta_1\cos\th_0+\al \beta_2\sin\th_0+\frac{\nu_1}{\beta_2}\frac{\tilde a^{22}}{\tilde a^{11}}(\al \nu_2\cos\th_0-\al \nu_1\sin\th_0)\big)\\
&\quad\qquad F_\al'(\th_0)\big(-\beta_1\sin\th_0+\beta_2\cos\th_0-\frac{\nu_1}{\beta_2}\frac{\tilde a^{22}}{\tilde a^{11}}(\nu_2\sin\th_0+\nu_1\cos\th_0)\big)\\
&\quad\qquad -\frac{1}{\epsilon}\frac{\nu_1}{\tilde{a}^{11}}\frac{\beta_1}{\beta_2}\frac{b_0^{2,1}}{\sin\th_0}F_\al(\th_0)\Big)\\
=&\,r^{\al-1}\Big(F_\al(\th_0)\big(-\al\be_0\cdot\tau-\al\frac{\nu_1}{\beta_2}\frac{\tilde a^{22}}{\tilde a^{11}}\big)-\be_0\cdot\nu F_\al'(\th_0)-\frac{1}{\epsilon}\frac{\beta_1}{\beta_2}\frac{b_0^{2,1}}{\tilde a^{11}}F_\al(\th_0)\Big).
\eeqas
By taking $\al>0$ sufficiently small, we recall the uniform limits as $\al\to0$, $F_\al(\th_0)\to1$ and $F_\al'(\th_0)\to0$. Thus we obtain the claimed inequality by observing that $\frac{\beta_1}{\beta_2}>0$.

Finally, a similar calculation shows that
$$ M_2v_\al\leq -c_1r^{\al-1}$$
provided first $\eps>0$ is sufficiently small and then $\al>0$ is sufficiently small.
\end{proof}

\begin{rmk}\label{rmk:barrier}
It follows from the construction of the Miller barrier as described above (and given in \cite[Chapter 3]{L}) that the dependence of $\al_0$ on $\th_0$ is monotone, and that also $v_\al$ depends continuously on $\al,\th_0$ in the admissible range. Therefore, given a fixed range $\th_0\in[\th_1,\th_2]$, one may fix a single $\al_0$ and use a single barrier $v_\al$ for the whole interval of $\th_0$, thereby making the constants $c_*$, $c_1$ uniform.
\end{rmk}

\section{Comparison Principle}\label{sec:comp}
In the proof of Theorem \ref{thm:main}, we made use of two versions of the comparison principle. We state and prove both of them together here. 
\begin{thm}\label{thm:comparison}
Let $L_0$ and $\be_0$ be the operator and boundary vector as in \eqref{def:frozen} and let $\tilde a\in\R$, $\tilde b>0$ be given. Suppose that $u\in C^2(\om)\cap C^1(\overline{\om}\setminus\{0\})\cap C(\overline{\om})$ satisfies either
\beq\label{eq:compare1}
\begin{cases}
L_0 u\leq0 & \text{ in }\om,\\
u_{y_2}=0 & \text{ on }\gsymm,\\
\be_0\cdot D u+\tilde{a}\tau\cdot D u-\frac{\tilde b}{y_2}u\leq0  & \text{ on }\gcone,\\
u\geq 0 & \text{ on } \overline{\gball},\\
u=0 & \text{ at } \{0\},
\end{cases}
\eeq
 or
\beq\label{eq:compare2}
\begin{cases}
L_0 u+\tilde{c}_0(y)u\leq0 & \text{ in }\om,\\
u\geq 0 & \text{ on }\overline{\gsymm\cup\gball},\\
\be_0\cdot D u+\tilde{a}\tau\cdot D u-\frac{\tilde b}{y_2}u\leq0  & \text{ on }\gcone,
\end{cases}
\eeq
where also $\tilde{c}_0(y)<0$ in $\om$.\\
Then $\inf u=0$. 
\end{thm}

\begin{proof}
First suppose that $u$ satisfies problem \eqref{eq:compare1}. We begin by returning to the domain $\Om$ by rotating around the $y_1$-axis. By \eqref{eq:L0bar}, the function $u$ then satisfies the uniformly elliptic equation obtained from \eqref{eq:fullspace} by freezing the principal coefficients at 0 and setting the lower order terms to be zero. By the strong maximum principle, $u$ cannot attain a minimum in $\Om$ unless it is constant. However, no negative constant will satisfy the oblique condition on $\gcone$ as $\tilde{b}<0$. Returning to the original domain, this implies that $u$ does not attain a negative minimum in $\om\cup\gsymm$.

Clearly $u$ cannot attain a negative minimum on either $\overline{\gball}$ or at $0$ by the Dirichlet conditions imposed there.

Finally, we check that $u$ does not attain a negative minimum on $\gcone$. If $u$ attains a negative minimum at $y^*\in \gcone$, then $\tau\cdot D u(y^*)=0$, $\beta_0\cdot Du(y^*)\geq0$ (as $\be_0$ is inward pointing). So $$\be_0\cdot D u(y^*)+\tilde{a}\tau\cdot D u(y^*)-\frac{\tilde b}{y^*_2}u(y^*)\geq-\frac{\tilde b}{y^*_2}u(y^*) >0,$$
contradicting the boundary condition.

In the second case, \eqref{eq:compare2}, we observe first that on the set where $u<0$, the partial differential inequality may be strengthened to the strict inequality $L_0u<0$. Hence, returning to the domain $\Om$, we have that $u$ satisfies this strict inequality for the uniformly elliptic operator $\overline{L}_0$, hence cannot attain a negative minimum in the interior of the set $\Om\cap\{u<0\}$. Moreover, by the boundary condition on $\overline{\gsymm\cup\gball}$, $u$ cannot attain a negative minimum on those portions of the boundary either. Finally, we follow the argument above to conclude that $u$ cannot attain a negative minimum on $\gcone$.
\end{proof}

\section{Counterexamples and Proof of Theorem \ref{thm:counterexample}}\label{sec:counter2}
We work on a domain $\Omega$, a cone with boundary $\Gcone$. For simplicity, we work in $\R^3$ and suppose that the axis of the cone is in the $x_3$ direction. In standard spherical coordinates $(r,\th,\vphi)$, we have that
$$\Omega=\{(r,\th,\vphi)\,|\,r>0,\,0\leq\th<\th_0,\,\vphi\in[0,2\pi)\},$$
where $\th_0\in(0,\pi)$. In this notation, the cylindrical symmetry coordinates for axi-symmetric functions correspond to $y_1=r\cos\th$, $y_2=r\sin\th$. We denote by $\LP_\al$ the Legendre polynomial of degree $\al$.
\begin{thm}
Consider the oblique derivative problem with axi-symmetric oblique vector
\beq\label{eq:purelaplace}
\begin{cases}
\Delta u=0 &\text{ in }\Om,\\
\be\cdot Du=0 &\text{ on }\Gcone,
\end{cases}
\eeq
where, in the (cylindrical) symmetry coordinates $(y_1,y_2)$, $\be=(\cos(s),\sin(s))$, $s\in(-\pi+\th_0,\th_0)$.
\begin{enumerate}
\item\label{itm:1} Suppose $\th_0\in(\frac{\pi}{2},\pi)$. Then there exists $s_0\in(-\pi+\th_0,0)$ such that if either 
$$s\in(-\pi+\th_0,s_0)\quad\text{ or }\quad s\in(\frac{\pi}{2},\th_0),$$
then there exists $\al=\al(\th_0,s)\in(0,1)$ such that 
$$u(r,\th)=r^\al\LP_\al(\cos\th)$$
is an axi-symmetric solution of \eqref{eq:purelaplace} that lies in $C^{0,\al}\setminus C^{0,\al+\epsilon}$ for any $\epsilon>0$. 
\item\label{itm:2} Suppose $\th_0\in(0,\frac{\pi}{2})$. Then there exists $s_0\in(-\frac{\pi}{2},0)$ such that if $s\in(-\frac{\pi}{2},s_0)$, then there exists $\al=\al(\th_0,s)\in(0,1)$ such that 
$$u(r,\th)=r^\al\LP_\al(\cos\th)$$
is an axi-symmetric solution of \eqref{eq:purelaplace} that lies in $C^{0,\al}\setminus C^{0,\al+\epsilon}$ for any $\epsilon>0$. 
\end{enumerate} 
\end{thm}

We note that in case (2), the obtained solution is strictly positive away from the origin. Numerics suggest that this is also the case for the solutions obtained for case (1).

\begin{proof}
We begin by seeking a separable, axi-symmetric solution of the Laplace equation, that is, a solution of the form
$$u(r,\th,\vphi)=R(r)\Theta(\th).$$
By standard ODE arguments, we then arrive at the function
$$u_\al(r,\th,\vphi):=r^{\al}\LP_{\al}(\cos\th),$$
where $\LP_\al$ is the Legendre polynomial of degree $\al$. This function is easily seen to satisfy the PDE of \eqref{eq:purelaplace} (see Appendix \ref{sec:counter1} below for more details on the derivation). To give a solution in full space, $u_\al$ must also satisfy the symmetry requirement that $\Theta'(0)=0$. This is easy to verify as 
$$\Theta'(0)=\lim_{\th\to0+}\big(-\sin\th\LP_\al'(\cos\th)\big)=\lim_{z\to1-}\Big(-\sqrt{1-z^2}\frac{(\al + 1) (z \LP_\al(z) - \LP_{\al + 1}(z))}{1-z^2}\Big)=0,$$
where we have applied the identity of \cite[(14.10.4)]{DLMF} for the derivative and \cite[(14.8.1)]{DLMF} for the limit.

We therefore search for $\al$ solving the oblique derivative condition. To that end, we move into the cylindrical coordinates $y=(y_1,y_2)$ and  apply again the identity of \cite[(14.10.4)]{DLMF},
$$\LP_\al'(z)=(\al+1)\frac{z\LP_\al(z)-\LP_{\al+1}(z)}{1-z^2},$$
to check
\beqas
(u_\al)_{y_1}= r^{\al-1}\big((2\al+1)\cos\th\LP_\al(\cos\th)-(\al+1)\LP_{\al+1}(\cos\th)\big),
\eeqas
and
\beqas
(u_\al)_{y_2}= r^{\al-1}\big(\sin\th\big(\al-(\al+1)\frac{\cos^2\th}{\sin^2\th}\big)\LP_\al(\cos\th)+(\al+1)\frac{\cos\th}{\sin\th}\LP_{\al+1}(\cos\th)\big).
\eeqas
As both of these derivatives separate, we define their angular components
 \beqas
  U_1(\th,\al)=&\,(2\al+1)\cos\th\LP_\al(\cos\th)-(\al+1)\LP_{\al+1}(\cos\th),\\
  U_2(\th,\al)=&\,\sin\th\big(\al-(\al+1)\frac{\cos^2\th}{\sin^2\th}\big)\LP_\al(\cos\th)+(\al+1)\frac{\cos\th}{\sin\th}\LP_{\al+1}(\cos\th).
  \eeqas
Solving the oblique condition is therefore equivalent to finding $\al$ such that 
\beqas
B(\th_0,\al,s):=&\,\cos(s)U_1(\th_0,\al)+\sin(s)U_2(\th_0,\al)=0.
\eeqas
 From the identities $\LP_0(z)=1$, $\LP_1(z)=z$, $\LP_2(z)=\frac{3z^2-1}{2}$ for all $z\in[-1,1]$, one easily checks that, for all $\th\in(0,\pi)$,
  $$\lim_{\al\to0+}U_1(\th,\al)=\lim_{\al\to0+}U_2(\th,\al)=0,$$
  and also
  $$\lim_{\al\to1-}U_1(\th,\al)=1,\quad \lim_{\al\to 1-}U_2(\th,\al)=0.$$
  Thus 
  $$B(\th_0,0,s)=0,\quad B(\th_0,1,s)=\cos(s),\quad\text{ for all }\th_0\in(0,\pi).$$
\textbf{Proof of \eqref{itm:1}.}\\
\textit{Case 1:} $s\in(-\pi+\th_0,0)$. Then we have that $B(\th_0,1,s)=\cos(s)>0$.
  We prove that there exists $s_0(\th_0)\in(-\pi+\th_0,0)$ such that for $s\in(-\pi+\th_0,s_0)$ we have that
  $$\frac{\d B}{\d\al}(\th_0,\al,s)\big|_{\al=0}<0.$$
 Assuming the existence of such an $s_0$, we clearly obtain for all $s\in(-\pi+\th_0,s_0)$ an $\al=\al(\th_0,s)\in(0,1)$ such that $B(\th_0,\al,s)=0$ as claimed in the theorem.
 
   We begin by computing the derivatives of $U_1$, $U_2$ with respect to $\al$. 
   \beqas
  \frac{\d U_1}{\d\al}=&\,2\cos\th_0\LP_{\al}(\cos\th_0)-\LP_{\al+1}(\cos\th_0)\\
  &+(2\al+1)\cos\th_0\d_\al\LP_{\al}(\cos\th_0)-(\al+1)\d_\al\LP_{\al+1}(\cos\th_0),\\
  \frac{\d U_2}{\d\al}=&\,\sin\th_0\big(1-\frac{\cos^2\th_0}{\sin^2\th_0}\big)\LP_{\al}(\cos\th_0)+\frac{\cos\th_0}{\sin\th_0}\LP_{\al+1}(\cos\th_0)\\
  &+\sin\th_0\big(\al-(\al+1)\frac{\cos^2\th_0}{\sin^2\th_0}\big)\d_\al\LP_{\al}(\cos\th_0)+(\al+1)\frac{\cos\th_0}{\sin\th_0}\d_\al\LP_{\al+1}(\cos\th_0).
  \eeqas
  From  \cite[(4.18)]{Sm06}, we have the identity
  \beq
  (\al+1)\frac{\d\LP_{\al+1}(z)}{\d\al}-(2\al+1)z\frac{\d\LP_\al(z)}{\d\al}+\al\frac{\d\LP_{\al-1}(z)}{\d\al}=-\LP_{\al+1}(z)+2z\LP_\al(z)-\LP_{\al-1}(z).
  \eeq
  Thus
  \beqas
  \frac{\d U_1}{\d\al}\big|_{\al=0}=&\,2\cos\th_0\LP_{0}(\cos\th_0)-\LP_{1}(\cos\th_0)\\
  &+\cos\th_0\d_\al\LP_{\al}(\cos\th_0)\big|_{\al=0}-\d_\al\LP_{\al+1}(\cos\th_0)\big|_{\al=0}\\
  =&\,\LP_{-1}(\cos\th_0)=1,
    \eeqas
    where we have used in the last line that $\LP_{-1}(z)=1$ for all $z$.
    
    Similarly,
    \beqas
     \frac{\d U_2}{\d\al}\big|_{\al=0}=&\,\sin\th_0\big(1-\frac{\cos^2\th_0}{\sin^2\th_0}\big)\LP_{0}(\cos\th_0)+\frac{\cos\th_0}{\sin\th_0}\LP_{1}(\cos\th_0)\\
     &-\frac{\cos^2\th_0}{\sin\th_0}\d_\al\LP_{\al}(\cos\th_0)\big|_{\al=0}+\frac{\cos\th_0}{\sin\th_0}\d_\al\LP_{\al+1}(\cos\th_0)\big|_{\al=0}\\
     =&\,\sin\th_0+\frac{\cos\th_0}{\sin\th_0}\big(-\LP_1(\cos\th_0)+2\cos\th_0\LP_0(\cos\th_0)-\LP_{-1}(\cos\th_0)\big)\\
     =&\,\sin\th_0+\frac{\cos\th_0}{\sin\th_0}(\cos\th_0-1)=\frac{1-\cos\th_0}{\sin\th_0}.
    \eeqas
    Hence we find
    \beq\label{eq:7.3}
    \frac{\d B}{\d\al}(\th_0,\al,s)\big|_{\al=0}=\cos(s)+\sin(s)\frac{1-\cos\th_0}{\sin\th_0}=:V(\th_0,s).
    \eeq
    One sees easily that for each fixed $\th_0\in(\frac{\pi}{2},\pi)$, $V$ is a strictly increasing function of $s$ on $(-\pi+\th_0,0)$ such that $V(\th_0,-\pi+\th_0)=-1$ and $V(\th_0,0)=1$. We choose $s_0(\th_0)$ to be the solution of $V(\th_0,s_0(\th_0))=0$ on this interval, so that for all $s\in(-\pi+\th_0,s_0)$, 
    $$\frac{\d B}{\d\al}(\th_0,\al,s)\big|_{\al=0}=V(\th_0,s)<0,$$
    as required.\\
    \textit{Case 2:} $s\in(\frac{\pi}{2},\th_0)$. 
    In this  case, we check from the above identities that $B(\th_0,0,s)=0$, $B(\th_0,1,s)=\cos(s)<0$ and, from the  formula \eqref{eq:7.3}, find also
     $$\frac{\d B}{\d\al}(\th_0,\al,s)\big|_{\al=0}=V(\th_0,s)>0\quad\text{ for all }\th_0\in(\frac{\pi}{2},\pi),\: s\in(\frac{\pi}{2},\th_0),$$ so that again there exists $\al=\al(\th_0,s)\in(0,1)$ such that $B(\th_0,\al,s)=0$, concluding the proof.\\
     \textbf{Proof of \eqref{itm:2}.}\\
     This proceeds in much the same manner as the above.
      It suffices to consider first that 
      $$B(\th_0,0,s)=0,\quad B(\th_0,1,s)=\cos(s)>0\text{ for }s\in(-\frac{\pi}{2},0),$$
     and again observe that the derivative 
      $$\frac{\d B}{\d\al}(\th_0,\al,-\frac{\pi}{2})\big|_{\al=0}=V(\th_0,-\frac{\pi}{2})<0\quad\text{ for all }\th_0\in(0,\frac{\pi}{2}),$$
      where $V(\th_0,-\frac{\pi}{2})<0$ follows directly from \eqref{eq:7.3}. 
      By continuity of $V$ with respect to $s$, we therefore obtain for each $\th_0\in(0,\frac{\pi}{2})$, an $s_0(\th_0)\in(-\frac{\pi}{2},0)$ such that for all $s\in(-\frac{\pi}{2},s_0)$, $V(\th_0,s_0)<0$, and conclude as before.
    \end{proof}

\appendix
\section{A counterexample to regularity in the absence of axi-symmetry}\label{sec:counter1}
Finally, we show that if the assumption of axi-symmetry is dropped, then the H\"older regularity may fail at the vertex. Such a result appears to be well-known among experts (see e.g.~\cite{KMR}), but we have been unable to find a rigorous proof of this fact.
We take $\Om$ and its boundary as in the previous section, with $\th_0\in(0,\pi)$, and we again work with the Laplace equation 
$$\Delta u=0 \text{ in }\Omega,$$
this time with the Neumann boundary condition,
$$\frac{\d u}{\d \nu}=0 \text{ on }\Gcone.$$
As we are no longer looking solely at axi-symmetric solutions, we write out more carefully the separation of variables. In spherical coordinates, the Laplacian takes the following form:
$$\Delta u=\frac{1}{r^2}\d_r(r^2 \d_r u)+\frac{1}{r^2\sin\th}\d_\th(\sin\th \d_\th u)+\frac{1}{r^2\sin^2\th}\d^2_{\vphi\vphi} u.$$
We make a separation of variables ansatz:
$$u(r,\th,\vphi)=R(r)\Theta(\th)\Phi(\vphi).$$
Then we obtain, after multiplying through by $r^2\sin^2\th /(R\Theta\Phi)$,
$$\frac{\sin^2\th}{R}\d_r(r^2R'(r))+\frac{\sin\th}{\Theta}\d_\th(\sin\th\Theta'(\th))+\frac{1}{\Phi}\Phi''(\vphi)=0.$$
Hence $\Phi^{-1}\Phi''(\vphi)$ is a constant. For periodic solutions, we need 
$$\Phi^{-1}\Phi''(\vphi)=-m^2,$$
where $m\in\Z$ (without loss of generality, $m\geq 0$). This gives us the obvious solution
$$\Phi(\vphi)=C\sin(m\vphi)+D\cos(m\vphi).$$
We therefore have axi-symmetry in the case $m=0$ (solution independent of $\vphi$) and do not have axi-symmetry in the case $m>0$. 
Substituting back in, we arrive at
$$\frac{\sin^2\th}{R}\d_r(r^2R'(r))+\frac{\sin\th}{\Theta}\d_\th(\sin\th\Theta'(\th))-m^2=0,$$
which is, after further separation of variables,
$$\frac{1}{R}\d_r(r^2R'(r))=\la,\quad \frac{1}{\sin\th \Theta}\d_\th(\sin\th\Theta'(\th))-\frac{m^2}{\sin^2\th}=-\la.$$
Solving for $R$, we obtain
$$R(r)=r^\al,\quad \la=\al(\al+1).$$ We are interested in solutions which are bounded at the vertex, hence we take $\al>0$, so $\la>0$ also. We must therefore study the following equation:
$$\frac{1}{\sin\th }\d_\th(\sin\th\Theta'(\th))+\big(\al(\al+1)-\frac{m^2}{\sin^2\th}\big)\Theta=0.$$
Defining $z=\cos\th$, we let $w(z)=\Theta(\th)$ and arrive at the associated Legendre equation:
\beq\label{eq:Legendre}
\frac{d}{dz}\big((1-z^2)\frac{d w}{dz}\big)+\big(\al(\al+1)-\frac{m^2}{1-z^2}\big)w=0.
\eeq
The solutions to equation \eqref{eq:Legendre} are the associated Legendre polynomials $\mathrm{P}^m_\al(z)$, well-defined for $z\in(-1,1)$ (i.e.~$\cos\th\in(-1,1)$, which is the range we need). 

For our purposes, it is sufficient to consider only the case $m=1$, i.e.~the first case without axi-symmetry. We then obtain solutions of Legendre's equation $\LP^1_\al(z)$ where $\al>0$. To solve the Neumann boundary condition, we look for $\al>0$ satisfying
$$(\LP^1_\al)'(\cos\th_0)=0.$$
If there exists a unique $\al\in(0,1)$ for each $\th_0\in(0,\pi)$ satisfying this equation, then
$$u(r,\th,\vphi)=r^{\al}\LP^1_{\al}(\cos\th)(C\sin\vphi+D\cos\vphi)$$
is a $C^{0,\al}$ solution that is not $C^{0,\al+\epsilon}$ up to the origin (provided at least one of $C$ and $D$ is non-zero). Such a solution is known in, for example, \cite[p.47]{KMR}. However, no analytic proof is given there, only a numerical plot. In the following, we therefore give a rigorous proof of the existence of such an $\al$.

We begin by recalling from \cite[(14.10.4)]{DLMF} that 
$$(\LP_\al^1)'(\cos\th_0)=\frac{-\al\LP_{\al+1}^1(\cos\th_0)+(\al+1)\LP_\al^1(\cos\th_0)}{\sin^2\th_0}=:W(\th_0,\al).$$
Moreover, from the recursive identity $\LP_\al^1(z)=-(1-z^2)^{\frac 1 2}\frac{d\LP_\al(z)}{dz}$ (e.g.~\cite[(14.6.1)]{DLMF}), we have
$$\LP_0^1(z)=0,\quad \LP_1^1(z)=-(1-z^2)^{\frac12},\quad\LP_2^1(z)=-3z(1-z^2)^{\frac 12}.$$
 Thus, evaluating $W$ at $\al=0,1$, we find
 $$W(\th_0,0)=0,\quad W(\th_0,1)=\frac{3\cos\th_0(1-\cos^2\th_0)^{\frac 12}-2\cos\th_0(1-\cos^2\th_0)^{\frac 12}}{\sin^2\th_0}=\frac{\cos\th_0}{\sin\th_0}<0.$$
 We now differentiate with respect to $\al$ to show that $\frac{\d W}{\d\al}|_{\al=0}>0$, thus allowing us to conclude. We begin by noting the following recurrence identity for derivatives of associated Legendre polynomials with respect to degree at integer order from \cite[(2.9)]{Sm11}: for $n,m\in\Z$,
 \beq
 (1-z^2)^{\half}\frac{\d\LP_\al^{m+1}}{\d\al}\big|_{\al=n}-(n-m)z\frac{\d\LP_\al^m(z)}{\d\al}\big|_{\al=n}+(n+m)\frac{\d\LP_\al^m(z)}{\d\al}\big|_{\al=n-1}=z\LP_n^m(z)-\LP_{n-1}^m.
 \eeq
 Applying this identity with $m=n=0$, we have 
 \beqas
 \frac{\d W}{\d\al}\big|_{\al=0}=&\,{\cosec}^2\th_0\Big(-\LP_1^1(\cos\th_0)+\cos\th_0\LP_0^1(\cos\th_0)+\cos\th_0\frac{\d\LP_\al^1(\cos\th_0)}{\d\al}\big|_{\al=0}\Big)\\
 =&\,{\cosec}^2\th_0\Big(\frac{\cos\th_0}{\sin\th_0}(\cos\th_0-1)-\LP_1^1(\cos\th_0)+\cos\th_0\LP_0^1(\cos\th_0)\big)\\
 =&\,{\cosec}^2\th_0\big(\frac{\cos\th_0}{\sin\th_0}(\cos\th_0-1)+\sin\th_0\big)\\
 =&\,{\cosec}^2\th_0\big(\frac{1-\cos\th_0}{\sin\th_0}\big)>0.
 \eeqas
 Thus, for each $\th_0\in(0,\pi)$, $W(\al,\th_0)$ has a root with respect to $\al$ in the interval $(0,1)$, and hence the Neumann problem for the Laplace equation admits a solution which is H\"older continuous but not Lipschitz up to the vertex.\\

\noindent\textbf{Acknowledgement.} The author wishes to thank Mikhail Feldman for his interest and advice during the writing of this paper.


\begin{thebibliography}{99}

\bibitem{BDL} \textsc{Barles, G. and Da Lio, F.}, Local $C^{0,\al}$ estimates for viscosity solutions of Neumann-type boundary value problems, \textit{J. Differential Equations} \textbf{225} (2006), 202--241.

\bibitem{CKL} \textsc{\v{C}ani\'c, S., Keyftiz, B.\,L. and Lieberman, G.M.}, A proof of existence of perturbed steady transonic shocks via a free boundary problem, \textit{Comm. Pure Appl. Math.} \textbf{53} (2000), 484--511.

\bibitem{CF} \textsc{Chen, G.-Q. and Feldman, M.,} \textit{The mathematics of shock reflection-diffraction and von Neumann's conjectures}, Annals of Mathematics Studies \textbf{197}, Princeton University Press, Princeton, NJ, (2018).

\bibitem{Fichera} \textsc{Fichera, G.}, On a unified theory of boundary value problems for elliptic-parabolic equations of second order, in \textit{Boundary problems in differential equations}, Univ. of Wisconsin Press, Madison, (1960), 97--120.

\bibitem{Finn} \textsc{Finn, R.}, \textit{Equilibrium capillary surfaces}, Grundlehren der mathematischen Wissenschaften \textbf{284}, Springer-Verlag, New York-Berlin-Heidelberg-Tokyo, (1986).

\bibitem{Fiorenza} \textsc{Fiorenza, R.}, Sui problemi di derivata obliqua per le equazioni ellittiche, \textit{Ricerche Mat.} \textbf{8} (1959), 83--110.

\bibitem{GH} \textsc{Gilbarg, D. and H\"ormander, L.,} Intermediate Schauder estimates, \textit{Arch. Rational Mech. Anal.} \textbf{74} (1980), 297--318.

\bibitem{GT} \textsc{Gilbarg, D. and Trudinger, N.}, \textit{Elliptic partial differential equations of second order}, Classics in Mathematics, Springer-Verlag, Berlin, (2001).

\bibitem{KN} \textsc{Kenig, C.\,E. and Nadirashvili, N.\,S.}, On optimal estimates for some oblique derivative problems, J. Funct. Anal. \textbf{187} (2001), 70--93.

\bibitem{KMR} \textsc{Kozlov, V.A., Maz'ya, V.G. and Rossman, J.}, \textit{Spectral problems associated with corner singularities of solutions of elliptic equations}, AMS Mathematical Surveys and Monographs \textbf{85}, Providence, RI, (2001).

\bibitem{L87} \textsc{Lieberman, G.\,M.}, Oblique derivative problems in Lipschitz domains I, \textit{Boll. Un. Mat. Ital. (7)} \textbf{1-B} (1987), 1185--1210.

\bibitem{L88} \textsc{Lieberman, G.\,M.}, Oblique derivative problems in Lipschitz domains II. Discontinuous boundary data, \textit{J. Reine Angew. Math.} \textbf{389} (1988), 1--21.

\bibitem{L01} \textsc{Lieberman, G.\,M.}, Pointwise estimates for oblique derivative problems in nonsmooth domains, \textit{J. Differential Equations} \textbf{173} (2001), 178--211.

\bibitem{L08} \textsc{Lieberman, G.\,M.}, Elliptic equations with strongly singular lower order terms, \textit{Indiana Univ. Math. J.} \textbf{57} (2008), 2097--2135.

\bibitem{L} \textsc{Lieberman, G.\,M.}, \textit{Oblique derivative problems for elliptic equations}, World Scientific Publishing Co. Pte. Ltd., Hackensack, NJ, (2013).

\bibitem{Michael} \textsc{Michael, J.\,H.}, Barriers for uniformly elliptic equations and the exterior cone condition, \textit{J. Math. Appl. Anal.} \textbf{79} (1981), 203--217.

\bibitem{Miller} \textsc{Miller, K.}, Barriers on cones for uniformly elliptic operators, \textit{Ann. Mat. Pura Appl.} \textbf{76} (1967), 93--105.

\bibitem{N} \textsc{Nadirashvili, N.\,S.}, On a problem with oblique derivative, \textit{Mat. Sb.} \textbf{127} (1985), 398--416. [English transl. \textit{Math. USSR Sb.} \textbf{55} (1986), 397--414.]

\bibitem{DLMF} \textsc{Olver, F.\,W.\,J., Lozier, D.\,W., Boisvert R.\,F. and Clark, C.\,W.} eds., \textit{NIST handbook of mathematical functions}, Cambridge University Press, (2010).

\bibitem{Sm06} \textsc{Szmytkowski, R.}, On the derivative of the Legendre function of the first kind with respect to its degree, \textit{J. Phys. A} \textbf{39} (2006), 15147--15172.

\bibitem{Sm11} \textsc{Szmytkowski, R.}, On the derivative of the associated Legendre function of the first kind of integer order with respect to its degree (with applications to the construction of the associated Legendre function of the second kind of integer degree and order), \textit{J. Math. Chem.} \textbf{49} (2011), 1436--1477.

\end{thebibliography}
 \end{document}